\documentclass[12pt]{iopart}

\expandafter\let\csname equation*\endcsname\relax
\expandafter\let\csname endequation*\endcsname\relax
\usepackage{amsmath}

\usepackage{amsfonts,amsthm}
\usepackage{multirow}
\usepackage{tikz,pgfplots}
\usepackage{pdfpages}
\usepackage{diagbox}
\usepackage{amsbsy}
\usepackage{amssymb}   
\usepackage{algorithm}
\usepackage[noend]{algpseudocode}
\usepackage{graphicx}  
\usepackage{pgf,tikz}
\usetikzlibrary{mindmap,trees,shapes,arrows,matrix}
\usepackage{verbatim}   
\usepackage{color}  
\usepackage{cite}    
\usepackage{subfigure}  
\usepackage{hyperref}
\hypersetup{
     colorlinks = true,
     linkcolor = green!50!blue,
     anchorcolor = blue,
     citecolor = blue,
     filecolor = blue,
     urlcolor = blue
     }

\usepackage{color}
\usepackage{booktabs}
\usepackage[margin=1in]{geometry}
\usepackage{array}
\usepackage{makecell}
\newcolumntype{x}[1]{>{\centering\arraybackslash}p{#1}}
\usepgfplotslibrary{fillbetween}

\newtheorem{prop}{Proposition}

\definecolor{observations}{rgb}{0.84,0.23,0.24}
\definecolor{titlecolor}{rgb}{0.08,0.38,0.74}
\newcommand{\bvec}[1]{\mathbf{#1}}
\newcommand{\bvecS}[1]{\ensuremath{\boldsymbol{#1}}}
\newcommand{\M}[1]{\mathbf{#1}}
\newcommand{\Mm}[1]{\ensuremath{\boldsymbol{#1}}}
\newcommand{\hF}[1]{\hat{F}}

\newcommand{\noise}{\sigma^{-2}}
\newcommand{\calH}{\mathcal{H}}

\newcommand{\Gpos}{\Mm{\Gamma}_{\!\mathrm{post}}}

\newcommand{\Gnoise}{\Mm{\Gamma}_{\!\mathrm{noise}}}
\newcommand{\Gnoiseinv}{\Mm{\Gamma}^{-1}_{\!\mathrm{noise}}}
\newcommand{\Gpr}{\Mm{\Gamma}_{\!\mathrm{pr}}}
\newcommand{\Gprsq}{\Mm{\Gamma}^2_{\!\mathrm{pr}}}
\newcommand{\Gprsqrt}{\Mm{\Gamma}^{\frac{1}{2}}_{\!\mathrm{pr}}}

\newcommand{\Cpr}{\mathcal{{C}}_{\!\mathrm{pr}}}
\newcommand{\Gprinv}{\Mm{\Gamma}^{-1}_{\!\mathrm{pr}}}
\newcommand{\mupr}{m_{\!\mathrm{pr}}}
\newcommand{\mupost}{m_{\!\mathrm{post}}}
\newcommand{\mpost}{\bvec{m}_{\!\mathrm{post}}}
\newcommand{\mpr}{\bvec{m}_{\!\mathrm{pr}}}

\newcommand{\Cpost}{\mathcal{C}_{\!\mathrm{post}}}
\newcommand{\R}{\mathbb{R}}
\newcommand{\Ns}{s}
\newcommand{\Nt}{r}

\newcommand{\phiN}{\phi_{\text{N}}}

\def\addressnyu{Courant Institute of Mathematical Sciences, New York University,
New York, NY, USA}
\def\addressncsu{Department of Mathematics, North Carolina State University, Raleigh, NC, 
USA}

\begin{document}

\title[OED under irreducible uncertainty]{Optimal experimental design
  under irreducible uncertainty for linear inverse problems governed by PDEs}


\author{Karina~Koval$^1$, Alen~Alexanderian$^2$ and Georg~Stadler$^1$}

\address{
$^1$\addressnyu\\
$^2$\addressncsu
}
\ead{\url{koval@cims.nyu.edu}, 
  \url{alexanderian@ncsu.edu},
  \url{stadler@cims.nyu.edu}}

\begin{abstract}
  We present a method for computing A-optimal sensor placements for
  infinite-dimensional Bayesian linear inverse problems governed by
  PDEs with irreducible model uncertainties. Here, irreducible
  uncertainties refers to uncertainties in the model that exist
  in addition to the parameters in the inverse problem, and that cannot be
  reduced through observations.  Specifically, given a statistical
  distribution for the model uncertainties, we compute the optimal
  design that minimizes the expected value of the posterior covariance
  trace. The expected value is discretized using Monte Carlo leading
  to an objective function consisting of a sum of trace operators and
  a binary-inducing penalty. Minimization of this objective requires
  a large number of PDE solves in each step. To make this problem computationally
  tractable, we construct a composite low-rank basis using a
  randomized range finder algorithm to eliminate forward and adjoint
  PDE solves. We also present a novel formulation  of 
  the A-optimal design objective that requires the trace of
  an operator in the observation rather than the parameter space. The
  binary structure is enforced using a weighted regularized
  $\ell_0$-sparsification approach. We present numerical results for
  inference of the initial condition in a subsurface flow problem with
  inherent uncertainty in the flow fields and in the initial times.
\end{abstract}



\section{Introduction}\label{sec:intro}
Many problems in the sciences and engineering require inference of
unknown/uncertain parameters from indirect observations and a mathematical
model that relates the parameters to these observations. Having access to informative data
is integral to accurate parameter inference.  However, the cost of data
collection or physical restrictions often limit how much data one can collect.
Even if one has access to large stores of data, processing large amounts of
it can be expensive; moreover, data can be full of redundancies---poor
experimental design choices can limit the information data contains about the
parameters. A natural question to ask is: how can we design experimental
conditions for data collection to optimally reconstruct/infer parameters of
interest?  Addressing this requires solving an optimal experimental design
(OED) problem.

Inverse problems stemming from real-world applications often contain
uncertainties in the governing model, in addition to the uncertain inversion
parameters.  We focus on \emph{irreducible} model uncertainties, i.e.,
uncertainties that---for all practical purposes---cannot be reduced via 
parameter estimation.  
Examples of such irreducible uncertainties arise in inversion of the initial
concentration of a contaminant in groundwater flow. In this problem, one
usually only has rough estimates of the true groundwater velocity
field. Additionally, one might not exactly know the exact time at
which the contaminant has been released.
Unlike uncertainties that could be reduced with observations and parameter
estimation, these uncertainties are difficult to reduce.
However, they should be taken into account when computing
experimental designs. This article is about design of experiments for inverse
problems governed by models containing such irreducible uncertainties.

We focus on infinite-dimensional Bayesian inverse problems in which the
parameter-to-observable map is linear; see section~\ref{sec:background}
for a brief overview. The irreducible uncertainty may enter
nonlinearly into the model and we assume some knowledge about its
distribution. This knowledge could be
based, for instance, on historic data, or it could be the posterior
distribution obtained as the solution of another Bayesian inverse problem.  To
accommodate such an irreducible additional uncertainty in OED, we extend the
notion of A-optimal design to A-optimal design under (irreducible) uncertainty,
which we define as the design that minimizes the expected value of the average
posterior variance; see section~\ref{sec:a-optimal}. 

In the present work, we restrict the idea of experimental design to that
of choosing locations for placing data collecting sensors, though our
approach can be adapted to more general experimental design problems. 
We formulate the OED problem as that of finding an optimal subset of locations for sensor placement
from a pre-specified network of candidate sensor locations. 
We assign to
each candidate location a non-negative \emph{design weight} indicating its importance.  
We seek binary
designs and interpret a candidate sensor location with a design weight of $1$
as a location where a sensor should be placed.  To circumvent the computational
challenges of binary optimization, we relax the binary condition and add a
sparsity-inducing (or binary-inducing) penalty to our optimal design objective; 
see section~\ref{subsec:oedProb}.

\textbf{Literature survey and challenges.}
Classical references for optimal experimental design problems include~\cite{AtkinsonDonev92,
Ucinski05,Pazman86,Pukelsheim93}.   
OED for inverse problems governed by computationally intensive models 
has been subject to intense research activity in the past couple of decades;
see, e.g.,~\cite{
KorkelKostinaBockEtAl04,
HaberHoreshTenorio10,HoreshHaberTenorio10,
HuanMarzouk13,LongScavinoTemponeEtAl13,
AlexanderianPetraStadlerEtAl16,
yu2018scalable,
ruthotto2018optimal,
attia2018goal}.
Our focus is on ill-posed infinite-dimensional Bayesian \emph{linear} inverse problems. 
We build on previous work~\cite{HaberHoreshTenorio08,HaberMagnantLuceroEtAl12,
alexanderian:oed}, which developed
efficient methods for computing A-optimal designs for
high- or infinite-dimensional linear inverse problems using either a Bayesian or a
frequentist approach. Other approaches for computing optimal experimental
designs for high/infinite-dimensional linear inverse problems
are explored, for example, in
\cite{neitzel_sparse,AlexanderianSaibaba18,alexanderian:Doptimal}. 
In \cite{neitzel_sparse}, the authors propose a measure-based
OED formulation that does not choose sensor locations from a finite
number of candidate locations but allows sensors to be placed 
anywhere on a closed subset of the domain.
The articles 
\cite{alexanderian:Doptimal,AlexanderianSaibaba18} explore an alternate OED
criterion---D-optimality---for infinite-dimensional Bayesian linear inverse problems.

We extend previous approaches by taking into account irreducible uncertainty in
the model. This results in a challenging optimization under uncertainty (OUU)
problem~\cite{Sahinidis04,BeyerSendhoff2007,kouri2018optimization}, as we now
describe. In linear inverse problems, the average posterior variance of the
inversion parameters---the A-optimal criterion---is defined by the trace
of the posterior covariance operator $\Cpost$. This operator is high-dimensional 
(upon discretization),
dense, and expensive-to-apply.  Specifically, applying $\Cpost$ to
vectors requires many PDE solves. This covariance operator is a function
of the vector of experimental design weights and the random variables
characterizing the model uncertainty.  The resulting OED under uncertainty
(OEDUU) problem is challenging, because the OEDUU
objective is the expected value of $\tr \left[ \Cpost \right]$, with expectation taken with
respect to the model uncertainty.  It is worth noting that optimizing  
$\tr \left[ \Cpost \right]$ even for a single instance of $\Cpost$ is in itself a challenging 
problem~\cite{alexanderian:oed}. 

As mentioned above, our goal is finding binary design vectors, which are in
general difficult to compute.  Our approach for this considers a
relaxation of the problem and uses an adaptation of the
regularized $\ell_0$-sparsification approach outlined
in~\cite{alexanderian:oed}.  Note that there are different
options to control sparsity of designs; see e.g.,~\cite{candes2008enhancing,
HaberMagnantLuceroEtAl12,yu2018scalable}.

\textbf{Our approach.}
We follow a sample average approximation (SAA) approach for the OEDUU problem, 
where we approximate the expectation in the OEDUU objective by sample averaging over the
model uncertainty. Corresponding to each realization of the irreducible model
uncertainty, we have a realization of the parameter-to-observable map (forward operator) that
defines a specific instance of $\Cpost$; see section~\ref{subsec:expected_valDisc}.
Computing traces of these operators is challenging 
due to their high-dimensionality (upon discretization) 
and the high cost (in terms of PDE solves) of applying the covariance operators 
to vectors. 

To mitigate the computational cost of OEDUU, we present a novel formulation of the
OED criterion in the observation space. Thus, we only require computing traces
of operators defined on the observation space which in many
infinite-dimensional inverse problems has a smaller dimension than that of the
discretized parameter space; see section~\ref{subsec:Approx_trace}.  However,
computing the resulting OEDUU objective and its gradient still requires many PDE
solves, and these computations are repeated in each step of an optimization
algorithm used for solving the OEDUU problem. Hence, it is imperative to
exploit problem structure to compute low-rank approximations of the forward
operator samples to eliminate frequent PDE solves from the optimization
iterations.  To do so, we employ randomized matrix methods to compute a low
rank basis that \emph{jointly} approximates the range space for all or subsets
(clusters) of forward operator samples; see section~\ref{sec:mod_red}.

We explore the effectiveness and performance of our methods for a
realistic groundwater initial condition inversion problem; see 
sections~\ref{sec:IP}--\ref{sec:numerics}.  In this application, the
(high-dimensional) irreducible uncertainties stem from: an unknown groundwater flow field (as a result of uncertainty in the subsurface permeability field), and from an uncertain observation time.

\textbf{Contributions.}
The contributions of this article are as follows: (1) We propose a mathematical 
formulation for OED under irreducible model uncertainty
in infinite-dimensional Bayesian linear inverse problems. 
(2) We present a novel OED objective
formulation in the observation space that avoids trace estimation in 
high-dimensional discretized parameter spaces. This formulation also
applies to A-optimal OED without additional model uncertainty.
(3) We develop an efficient and practical  reduced order modeling 
framework for OEDUU 
that eliminates PDE solves from the optimization process.
(4) We present a comprehensive set of numerical experiments that illustrate 
the proposed approach and demonstrate its effectiveness. 

\textbf{Limitations.}  The present work also has
limitations. (1) Our formulation is restricted to linear
parameter-to-observation maps, Gaussian priors, and additive Gaussian
noise. Further work is needed to extend it
to nonlinear inverse problems.  One possible extension is to use a
Gaussian approximation of the posterior distribution as in
\cite{AlexanderianPetraStadlerEtAl16}.  (2) We use Monte Carlo
sampling to approximate the irreducible uncertainty. If the aim is a
highly accurate approximation of this uncertainty, a large number of
samples might be needed.

\section{Background}\label{sec:background} In this section, we review
relevant material required for the formulation of OED problems under
uncertainty for infinite-dimensional Bayesian inverse problems. We
also summarize previous work which this paper builds on. 

\subsection{Infinite-dimensional Bayesian linear inverse problems}\label{subsec:bayes_inf}

We begin our discussion by formulating a prototypical Bayesian linear inverse
problem, which is the primary focus of the present work.  A detailed treatment
of general nonlinear Bayesian inverse problems can be found for instance
in~\cite{dashti:bayesian}.
 
Given finite-dimensional observations, $\bvec{d} \in \mathbb{R}^{d}$, we seek
to infer an unknown parameter, $m$, which is related to the data
through
\begin{equation}\label{eq:paramToObs} \centering \bvec{d} =
  \mathcal{F}m + \bvec{\eta}.
\end{equation}
We consider the case where $m$ is an element of an
infinite-dimensional Hilbert space $\mathcal{H}$, e.g.,
$\mathcal{H} = L^2(\mathcal{D})$ with
$\mathcal{D}$ being a bounded domain in $\R^2$ or $\R^3$.  Here,
$\mathcal{F}:\mathcal{H} \rightarrow \mathbb{R}^{d}$ is a continuous linear
parameter-to-observable map (forward model). In our target
applications, computing $\mathcal{F}m$ for a given $m$ involves
solving a partial differential equation (PDE) followed by 
application of an observation operator.  In~\eqref{eq:paramToObs}, we
assume $\bvec{\eta} \sim \mathcal{N}(\bvec{0},\Gnoise)$.

Following a Bayesian approach, we model the parameter $m$ as a random variable
and impose a prior probability law for $m$. The prior law is a
probabilistic description of our prior knowledge about the parameter.  We use a
Gaussian prior $\mathcal{N}(\mupr,\Cpr)$ where the mean
$\mupr$ is a sufficiently regular element of $\mathcal{H}$ and the covariance
operator $\Cpr$ is a trace-class operator defined through the inverse of a
differential operator. More precisely, we let $\Cpr = \mathcal{A}^{-2}$ where
$\mathcal{A} = -\rho \Delta + \delta \mathcal{I}$; here, $\Delta$ is
the Laplace operator, $\delta > 0$
controls the magnitude of the variance and $\rho > 0$ controls the correlation
length.  This choice ensures that the prior covariance operator is trace-class
in two and three space dimensions; see \cite{dashti:bayesian} for details.

The solution of the Bayesian inverse problem is the posterior probability law
for the parameter, which is conditioned on measurement data. 
For a linear inverse problem with a Gaussian prior and an 
additive Gaussian noise model, 
which is what we assume, it is well known~\cite{dashti:bayesian} 
that the posterior distribution is also 
Gaussian, namely
$\mathcal{N}(\mupost,\Cpost)$ with  
\begin{equation}\label{eq:inf_bayes}
\Cpost = \left(\mathcal{F}^*{\Gnoiseinv}\mathcal{F}+\Cpr^{-1} \right)^{-1} \quad 
\text{and} \quad \mupost = \Cpost\left(
\mathcal{F}^*{\Gnoiseinv}\bvec{d}+\Cpr^{-1}\mupr \right).
\end{equation}

\subsection{Bayesian linear inverse problem with model uncertainty}\label{subsec:bayes_uncertain}
We are interested in the design of experiments for inverse problems governed by
PDEs with uncertain parameters representing the irreducible
uncertainty. Note that these
uncertainties are in addition to the uncertainty in the inversion parameter.
In such cases, the forward operator $\mathcal{F}$ is a function of
uncertain parameters. We will formalize this below.

Let $(\Omega,\mathcal{G},P)$ be a probability space, where $\Omega$ is a sample
space, $\mathcal{G}$ is a suitable $\sigma$-algebra on $\Omega$, and $P$ is a
probability measure.  We let $\xi = \xi(\omega)$ denote a random variable,
defined on $(\Omega, \mathcal{G}, P)$, that models the irreducible uncertainty in
$\mathcal{F}$; for example, $\xi$ could be a random vector whose entries define
uncertain parameters in the governing PDEs, or $\xi$ could be a function-valued
random variable (i.e., a random field coefficient).  In this case, the forward
operator is a random variable $\mathcal{F} : \Omega \rightarrow
\mathcal{L}(\mathcal{H}, \R^{d})$, where $\mathcal{L}(\mathcal{H},\R^{d})$ is
the space of linear transformations from $\mathcal{H}$ to $\R^d$.
For each $\omega \in \Omega$, we have a realization $\mathcal{F}(\omega)
= \mathcal{F}(\xi(\omega))$ of the forward model.
Thus, the solution of the corresponding 
Bayesian linear inverse problem depends on $\xi$ and is 
given by 
$\mathcal{N}(\mupost(\xi),\mathcal\Cpost(\xi))$, with 
\begin{equation}\label{eq:inf_bayesU}
\Cpost(\xi) = \left(
\mathcal{F}(\xi)^*{\Gnoiseinv}\mathcal{F}(\xi)+\Cpr^{-1} \right)^{-1}
\quad \text{and}
\quad \mupost(\xi) = \Cpost\left(
\mathcal{F}(\xi)^*{\Gnoiseinv}\bvec{d}+\Cpr^{-1}\mupr\right).
\end{equation}

\subsection{A-optimal design of infinite-dimensional inverse problems}
We focus on A-optimal design of Bayesian linear inverse problems. That is,
we seek designs (sensor placements) that result in minimized average posterior
variance.  In the present infinite-dimensional formulation, this amounts to
minimizing the trace of $\Cpost$; see \cite{alexanderian:oed} for details. 
The additional challenge for inverse problems with uncertainties in 
the governing model is that the covariance operator itself depends
on the random variable $\xi$ that defines the uncertain parameters in 
the model. In the next section, we formulate the A-optimal design problem
as that of optimization under uncertainty. We refer to this problem 
as optimal experimental design under uncertainty (OEDUU).

\section{A-optimal design of experiments for linear Bayesian inverse problems
with irreducible model uncertainty}\label{sec:a-optimal} 
We begin with a discussion of experimental
design for infinite-dimensional linear inverse problems with uncertain forward models in
section~\ref{subsec:design}. In section~\ref{subsec:AoptimalDesign}, we
present our 
formulation of Bayesian A-optimality for such inverse problems, and in
section~\ref{subsec:oedProb}, we formulate the optimization problem for finding
A-optimal designs in the inverse problems under study. 

\subsection{Design in the Bayesian inverse problem with uncertain forward models}
\label{subsec:design}
      
In OED for inverse problems we are interested in determining how to
collect measurement data to optimize the parameter inference. The definition of
an ``experimental design'' is problem specific. For example, in inverse problems
in tomography, the design could correspond to choosing a subset of angles for
an x-ray source to hit an object. On the other hand, in the inverse problem
of identifying the source of a contaminant, the design corresponds to the  
placement of sensors that are  
used to measure the contaminant concentration. 

We consider a finite generic set of candidate experiments denoted by $x_i \in X$, $i =
1,\ldots,d$, where $X$ is a problem specific set of admissible experiments. More concretely, in the tomography example, $x_i$'s are measurement angles and $X$
corresponds to the set of all possible measurement angles; and in the
subsurface flow example, $x_i$'s indicate sensor locations and $X$ is the
physical domain in which the sensors can be placed.  
We assign a nonnegative weight $w_i$ to each choice
of $x_i$. Thus, an experimental design is specified
by a design vector $\bvec w = (w_1,w_2,\ldots,w_{d})$.
Ideally, we would like binary
design vectors: in the case $w_i = 1$,
we will collect the measurement corresponding to $x_i$, 
and if $w_i = 0$, the corresponding experiment will not be performed
(or the data not be collected).
However, an OED problem of finding binary optimal design vectors
has combinatorial complexity---an extremely challenging problem. 
To cope with this, 
we relax the binary assumption on the weights and allow the weights to
take values in the interval $[0,1]$. 
We then seek to enforce a binary structure
on the computed weights using a suitable penalty method, as discussed 
further below.

To incorporate a generic design $\bvec w$ into the Bayesian inverse problem, we
define a diagonal matrix $\M{W} \in \R^{d \times d}$, which, 
in a generic OED problem, contains the weights $\bvec w = (w_1, w_2, \ldots,
w_{d})$ on the diagonal.  In a sensor placement problems for an inverse
problem governed by a time-dependent PDE, as in the example
used in section~\ref{sec:IP},
the goal is to select an optimal subset of $\Ns$ candidate sensor locations, 
which collect measurements at $\Nt$ observation times. 
When a sensor location is chosen,  we assume that it collects
measurements for all $\Nt$ observation times.
In this setup, the design vector $\bvec{w}$ has dimension $\Ns$, and 
the vector of measurement data has dimension $d = \Ns\Nt.$
Thus, $\M{W} \in \R^{\Ns\Nt \times \Ns\Nt}$ is a block
diagonal matrix with each $\Ns \times \Ns$ diagonal block being a diagonal
matrix with the sensor weights on the diagonal. 

We incorporate the design vector $\bvec w$ in the Bayesian inverse problem
by considering a \emph{weighted forward operator} 
${\mathcal{F}}(\xi;\bvec{w}) := \M{W}^{\frac{1}{2}}\mathcal{F}(\xi)$.  
As before, $\xi$ is the random variable that models uncertainty in the 
governing PDEs.
The posterior law of $m$ now depends on the design 
$\bvec{w}$ as follows: 
\begin{equation}\label{eq:weight_cov}
\begin{aligned}
&\Cpost(\xi,\bvec{w}) = \left(
{\mathcal{F}}^*(\xi)\M{W}^{\frac{1}{2}}{\Gnoiseinv}\M{W}^{\frac{1}{2}}{\mathcal{F}}(\xi)+\Cpr^{-1}
\right)^{-1} \quad \text{and} \\
&\mupost(\xi,\bvec{w}) = \Cpost\left(
\mathcal{F}(\xi)^* \M{W}^{\frac12}\Gnoiseinv\bvec{d}+\Cpr^{-1}\mupr\right).
\end{aligned}
\end{equation}
As discussed later, if the noise covariance $\Gnoiseinv$ is diagonal, 
the expression for the posterior covariance operator simplifies and it
is not necessary to consider the square root of $\M{W}$.

\subsection{A-optimal design under uncertainty}\label{subsec:AoptimalDesign}
We focus on A-optimal design of linear inverse problems. 
Following
\cite{alexanderian:oed}, for a fixed realization of $\xi$, 
the A-optimal design is one which minimizes
the average posterior variance, i.e., the design which minimizes
$\tr{\left[ \Cpost(\xi,\bvec{w})\right]}$. 
We extend the notion of A-optimal designs to Bayesian inverse problems with
uncertain model parameters by formulating the OED problem as that of
minimizing the expected value of $\tr\left[\Cpost(\xi(\omega),\bvec{w})\right]$. 
Thus, the OED criterion we consider is given by 
\begin{equation}\label{eq:expected_val}
\phi(\bvec{w}) := \int_{\Omega} \tr\left[\Cpost(\xi(\omega),\bvec{w})\right] P(d\omega).
\end{equation}
Note that taking the perspective of optimization under uncertainty,
minimizing~\eqref{eq:expected_val} amounts to a risk-neutral
design. Other scalarization methods are possible, e.g., risk-averse
objectives, which place emphasis on avoiding particularly poor designs
\cite{kouri2018optimization, ShapiroDentchevaRuszczynski09}.

\subsection{The OED Problem}\label{subsec:oedProb}
An effective solution method for finding OEDs must 
provide the user a mechanism to 
strike a balance between the competing goals of minimizing 
posterior uncertainty and
using as few sensors as possible. We address this 
by using a sparsifying penalty function
to promote sparse (and eventually, binary) optimal design vectors. Accordingly, we formulate the
optimization problem for finding an OED as follows:
\begin{equation}\label{eq:oed_functional}
\min_{\bvec w \in [0, 1]^\Ns} \phi(\bvec{w})  
+ \gamma \psi \left( \bvec w \right), 
\end{equation}
where $\phi(\bvec{w})$ is the OED criterion defined in~\eqref{eq:expected_val}, 
$\psi: \mathbb{R}_{+}^{\Ns} \rightarrow [0,\infty)$ is a sparsity-inducing penalty
function, and $\gamma$ controls the degree of sparsity. 

A simple choice for $\psi$ is the $\ell_1$-norm, $\psi(\bvec w) =
\bvec 1^T\bvec w$, where $\bvec{1}$ is the vector of all ones. While using an $\ell_1$-norm penalty leads to sparse
designs, it does not yield binary design vectors. To obtain
binary designs, using an $\ell_1$-norm penalty
can be combined with thresholding
to decide where the
sensors should be placed, but as numerically studied in
\cite{alexanderian:oed}, the resulting designs are suboptimal.

In the present work, we follow the regularized 
$\ell_0$-sparsification procedure introduced
in~\cite{alexanderian:oed}, which typically
results in binary optimal design vectors.  
This approach involves 
a continuation procedure in which a sequence of minimization problems
with penalty functions that successively approximate the $\ell_0$-``norm'' 
are solved.  For $\bvec w \in \mathbb{R}^{\Ns}$, we denote the $\ell_0$-``norm'' as $\| \bvec w \|_{\ell_0}$ and define it as the number of non-zero entries in $\bvec w$. 
The procedure is initialized with
$\bvec{w}_{\varepsilon(0)}$, which is obtained by solving
\eqref{eq:oed_functional} with an ${\ell_1}$ penalty. Then, at step $i$
of the continuation procedure, a new solution $\bvec w_{\varepsilon(i)}$ is
obtained by solving 
\begin{equation}\label{eq:sparsification_func} 
\min_{\bvec w \in [0,1]^{\Ns}}
{\phi}(\bvec{w}) + \gamma \psi_{\varepsilon(i)}
(\bvec w) 
\end{equation}
using the previous $\bvec w_{\varepsilon(i-1)}$ as an initial guess; the
penalty function $\psi_{\varepsilon(i)}(\bvec w)$ is chosen from a family of
continuously differentiable penalty functions that approach the
$\ell_0$-``norm'' as $i \rightarrow \infty$.
This procedure uses non-convex functions and thus one cannot
guarantee uniqueness of solutions beyond the $\ell_1$-norm
initialization. However, numerical experiments in
\cite{alexanderian:oed} and in section \ref{sec:numerics} show that
the method performs well in practical examples.

We find that using the $\ell_0$ sparsification approach, there  may be
a large disparity between the sparsity (and the values of the
objectives) for the $\ell_1$ solution,
$\bvec{w}_{\varepsilon(0)}$, and for the binary weight vector $\bvec{w}_*$ obtained
at the end of the continuation procedure. That is, for a fixed $\gamma$, we found that $\| \bvec{w}_{\varepsilon(0)} \|_{\ell_1} \ll \| \bvec{w}_* \|_{\ell_0}$. This results in the
continuation procedure being rather sensitive to the choice of the
sequence $\varepsilon(i)$. As a remedy, we use a rescaling to help
mitigate this issue, as follows.  Note we can scale the $\ell_1$ penalty linearly
with a scaling factor $\alpha > 0$ through $\| \alpha \bvec{w} \|_{\ell_1} =
\alpha \| \bvec{w} \|_{\ell_1}$.  However, this scaling has no effect on the
$\ell_0$ norm $\| \alpha \bvec{w} \|_{\ell_0} = \| \bvec{w} \|_{\ell_0}$.
Since the first step in the continuation strategy in \cite{alexanderian:oed}
involves solving the minimization problem with an $\ell_1$ penalty to obtain
$\bvec{w}_{\varepsilon(0)}$, we can control the sparsity of
$\bvec{w}_{\varepsilon(0)}$ by using an $\alpha$-scaled $\ell_1$ penalty. More
specifically, $\alpha \in (0,1)$ produces a less sparse initial guess for the
continuation procedure than $\alpha \ge 1$. Once an initial guess
$\bvec{w}_{\varepsilon(0)}$ is obtained, we use the following family of scaled
sparsity-inducing penalty functions:
\begin{equation}\label{eq:weightedPenalty}
\psi_{\varepsilon}(\alpha \bvec{w}) := \sum_{i=1}^{\Ns} f_{\varepsilon}(\alpha
w_i),
\end{equation} 
where, for $\varepsilon > 0$, 
\begin{equation}\label{eq:sparsity_penalty}
 f_{\varepsilon}(\alpha w) = 
\begin{cases}
\alpha w/\varepsilon \quad &\mbox{if  }\hspace{3mm} w < \frac{\varepsilon}{2\alpha}, \nonumber\\
c_3 {(\alpha w)}^3 + c_2 {(\alpha w)}^2 + c_1 (\alpha w) + c_0 \quad &\mbox{if  }\hspace{3mm} \frac{\varepsilon}{2\alpha} \leq  w < \frac{2\varepsilon}{\alpha}, \nonumber\\
1 \quad &\mbox{if  }\hspace{3mm} w \geq \frac{2\varepsilon}{\alpha}; 
\end{cases}
\end{equation}
as in~\cite{alexanderian:oed}, the constants $c_3,c_2,c_1$ and $c_0$ are chosen such that
$f_{\varepsilon}(\alpha w)$ is continuously differentiable on $[0,1]$. 

To summarize, we choose a decreasing sequence $\{ \varepsilon(i)
\}_{i=1}^{\infty}$ such that $\varepsilon(1) < 1$ and $\varepsilon(i)
\rightarrow 0$ as $i \rightarrow \infty$. For a fixed $\gamma$, we initialize
the procedure with $\bvec{w}_{\varepsilon(0)}$, the solution to
\eqref{eq:sparsification_func} with an $\alpha$-weighted $\ell_1$ penalty, and
in each subsequent step $i$ of the procedure we minimize
\begin{equation}\label{eq:finalOEDfunctional} 
{\phi}(\bvec{w}) + \gamma \psi_{\varepsilon(i)}(\alpha \bvec{w}), 
\end{equation}
using $\bvec{w}_{\varepsilon({i-1})}$ as the initial iterate, until we converge to a
binary solution, $\bvec w_*$.  In practice, the choice of $\alpha$ is
problem-specific and we heuristically choose it such that $ \|
\bvec{w}_{\varepsilon(0)} \|_{\ell_1} \approx \| \bvec{w}_*
\|_{\ell_0}$, where $\bvec{w}_{\varepsilon(0)}$ is the non-binary weight vector obtained using the penalty function $\| \alpha \bvec w \|_{\ell_1}$, and $\bvec w_*$ is the binary weight-vector obtained following the continuation procedure with initial guess $\bvec{w}_{\varepsilon(0)} $. This  might require solving the problem with a few
choices of $\alpha$.

\section{Finite-dimensional approximation of OED objective and its gradient}\label{sec:discrete_oed}
In this section, we describe the discretization of the OED problem. This includes
the discretization of the operators defining 
the OED objective and gradient, and the 
approximation of the expected value in the OED objective in~\eqref{eq:oed_functional}
(see section~\ref{subsec:expected_valDisc}). We also present novel efficient-to-evaluate 
expressions for the OED objective and its gradient by taking traces of
operators on the measurement space (see section~\ref{subsec:Approx_trace}). The latter provides computational 
advantages in cases where the measurement dimension is significantly smaller than the
discretized parameter dimension.

\subsection{The discretized OED problem}\label{subsec:expected_valDisc}
Henceforth, we assume that the forward operator has been discretized
in space and, if applicable, in time.
That is, we consider $\M{F} : (\Omega,\mathcal{G},P) \rightarrow
\mathcal{L}(V_h, \R^d)$ where $V_h$ is a finite-dimensional subspace of $\calH$ given
by the span of $n$ finite-element nodal basis functions  $\{\varphi_i\}_{i=1}^n$. 
For each realization of $\xi$, which parameterizes model uncertainty, 
$\M{F}(\xi)$ is a linear transformation from $V_h$ to the measurement space $\R^d$. 
The discretized inversion
parameter is given by $m_h = \sum_{i=1}^n {m}_i\varphi_i$. Thus, instead of
inferring the probability law for our random function $m$, we focus on
characterizing the posterior distribution for the vector of coefficients,
$\bvec m = ({m}_1 \,\, {m}_2 \, \cdots \,\, {m}_n)^T$. 

The discretized parameter space is $\R^n$ equipped with the so called mass-weighted inner product
that approximates the $L^2(\mathcal{D})$ inner product.
The latter is the Euclidean inner product weighted by the finite element mass matrix;
see~\cite{bui:infBayes} for details.
Thus, the discretized forward
operator is $\M{F}(\xi): \R^{n} \rightarrow \R^{d}$. 
In this setting, 
the adjoint $\M{F}^*(\xi)$
of $\M{F}(\xi)$, is given by $\M{F}^{*}(\xi) = \M{M}^{-1}\M{F}^T(\xi)$ where
$\M{M}$ is the mass matrix. 

In what follows, we assume the observations are corrupted by
uncorrelated Gaussian noise with a constant variance
of $\sigma^2$; that is, $\Gnoise = \sigma^2 \M{I}$. 
The posterior distribution of the discretized parameter $\bvec m$ 
is then given by $\mathcal{N}(\mpost(\xi,\bvec{w}),\Gpos(\xi,\bvec{w}))$ with 
\begin{equation}\label{eq:disc_dist}  
\begin{aligned}
\Gpos(\xi,\bvec{w}) &= \left(\noise\M{F}^*(\xi)\M{W}\M{F}(\xi)+\Gprinv
\right)^{-1} \quad \text{and}
\\
\mpost(\xi,\bvec{w}) &= \Gpos\left(\noise
\M{F}(\xi)^*\M{W}^{\frac{1}{2}}\bvec{d}+\Gprinv \mpr\right). 
\end{aligned}
\end{equation}
Here, $\Gpr$ and $\mpr$ are the discretizations of $\Cpr$ and $\mupr$, respectively.

We follow a sample average approximation (SAA) approach for
solving~\eqref{eq:oed_functional}.  Thus, the expectation in the OED objective
is approximated via sample averaging.  Possibilities include Monte Carlo
(MC) sampling, quasi-Monte Carlo, or quadrature.  In the present work, we rely
on MC.
Let $\xi_i$, $i = 1,\ldots,N$, be realizations of $\xi$, and 
let $\M{F}_i = \M{F}(\xi_i)$. We approximate \eqref{eq:expected_val} with 
\begin{equation}\label{eq:finite_expectation}
\phi(\bvec{w}) \approx \bar{\phi}_N(\bvec{w}) :=
\frac{1}{N}\sum_{i=1}^N
\tr\left[\left(\noise\M{F}_i^*\M{W}\M{F}_i+{\Gprinv}\right)^{-1}\right].
\end{equation}
Note that the sample set $\{ \xi_i\}_{i=1}^N$ will be fixed in the
optimization problem. 

Additionally, we note that the map $\bvec w \mapsto\bar{\phi}_N(\bvec w)$ is
convex. Since $\Gpos(\xi, \bvec w)$ is a selfadjoint positive definite operator for
any fixed $\xi \in \Omega$ and $\bvec w \in \mathbb{R}^d$, the convexity of
$\bar{\phi}_N$ follows from (1) the strict convexity of $\M X \mapsto \tr
\left[ \M X^{-1} \right]$ on the cone of selfadjoint positive definite operators
(see~\cite{Pazman86}), (2) 
the mapping $\bvec w \mapsto \left[\Gpos(\xi,\bvec w)\right]^{-1} $ 
being affine (for fixed $\xi$), and
(3) the finite sum of convex functions being convex. To argue strict convexity
of $\bar\phi_N$, we need an additional assumption on the affine map $\bvec w
\mapsto \left[\Gpos(\xi,\bvec w)\right]^{-1}$, that is, we require that for any
$\bvec w_1, \bvec w_2 \in \mathbb{R}^d$ ($\bvec w_1 \not = \bvec w_2$), $
\M{F}_i^*\M{W}_1\M{F}_i \not =  \M{F}_i^*\M{W}_2\M{F}_i$, for at least 
one $i \in \{1, \ldots, N\}$. This will not hold,  
for example, if no information about the parameter $\bvec m$
could be learned from data obtained at one or more sensor locations.

\subsection{Efficient computation of OED objective and its gradient}\label{subsec:Approx_trace}

In~\eqref{eq:finite_expectation}, each term in the summation requires computing
the trace of the inverse of an operator whose
dimension is determined by the discretized parameter dimension $n$. 
For inverse problems governed by PDEs in two and three space dimensions, $n$ is 
typically very large. Thus, 
computing these traces directly is infeasible and methods
based on low-rank spectral decomposition or randomized trace estimation must be
employed~\cite{alexanderian:oed}.  Here, we 
outline an alternate strategy and present a reformulation of
the OED objective that involves computing traces of operators defined on the
measurement space.  In problems where the measurement dimension $d$ is
considerably smaller than the discretized parameter dimension $n$, this approach
provides significant computational savings, in particular in
combination with low-rank approximation of the (preconditioned)
parameter-to-observable map as in \cite{alexanderian:oed}, which we
generalize in section~\ref{sec:mod_red} to accommodate additional model
error.  Moreover, while $n$ grows upon mesh
refinement, $d$ remains fixed due to finite-dimensionality of the
observations. 

The following result facilitates our proposed reformulation of the A-optimal
OED objective. We state the result for a generic 
forward operator $\M{F}$.
\begin{prop}\label{prp:posterior_cov}
The following relation holds:
\begin{equation}\label{equ:posterior_cov}
(\noise\M{F}^* \M{W} \M{F} + \Gpr^{-1})^{-1}  = \Gpr - \noise\Gpr\M{F}^*(\M{I} + 
\noise\M{W} \M{F} \Gpr\M{F}^*)^{-1}\M{W} \M{F} \Gpr.
\end{equation}
\end{prop}
\begin{proof}
The result follows by the Sherman-Morrison-Woodbury identity
\cite[2.1.4]{golub2012matrix} , which states that for matrices $\M{A}
\in \mathbb{R}^{n \times n}$ and $\M{U},\M{V} \in \mathbb{R}^{n \times
  d}$,  $\M{A} + \M{U}\M{V}^T = \M{A}^{-1} -
\M{A}^{-1}\M{U}\left(\M{I} +
\M{V}^T\M{A}^{-1}\M{U}\right)^{-1}\M{V}^T\M{A}^{-1}$ provided
$\left(\M{I} + \M{V}^T\M{A}^{-1}\M{U}\right)$ and $\M{A}$ are
invertible. Setting $\M{A} := \Gpr^{-1}, \M{U} := \noise\M{F}^*$ and
$\M{V}^T :=  \M{W} \M{F}$, it is thus sufficient to prove the invertibility of the matrix $\M{I} + \M{W} \M{F} \Gpr\M{F}^*$.
 
With no loss of generality, we assume $\sigma^2 = 1$ for simplicity.
In the present setup, we have $\Gpr^* = \M{M}^{-1} \Gpr^T \M{M} = \Gpr$,
where $\M{M}$ is symmetric positive definite. Moreover, we have that
$\M{F}^* = \M{M}^{-1} \M{F}^T$ and thus
\[
(\M{F} \Gpr\M{F}^*)^T 
= (\M{F}^*)^T \Gpr^T \M{F}^T
= \M{F} \M{M}^{-1} \Gpr^T \M{M} \M{F}^*
= \M{F} \Gpr \M{F}^*.  
\] 
Hence, $\M{F} \Gpr\M{F}^*$ is symmetric; it is also clearly positive semidefinite.
Since $\M{W}$ and $\M{F} \Gpr\M{F}^*$ are both symmetric 
positive semidefinite matrices, their product 
$\M{W} \M{F} \Gpr\M{F}^*$
has nonnegative eigenvalues. Hence,
$\M{I} + \M{W} \M{F} \Gpr\M{F}^*$ is invertible, which ends the proof.
%
\end{proof}

Using this result, we have, for all $i \in \{1, \ldots, N\}$,
\begin{equation}\label{eq:sherman_morrison} \Gpos(\xi_i,\bvec{w}) = \Gpr - \noise\Gpr\M{F}^*_i \M{S}^{-1}_i(\bvec{w}) \M{W}\M{F}_i\Gpr,
\end{equation}
where $\M{S}_i(\bvec{w}) := \left( \M{I} + \noise\M{W}\M{F}_i\Gpr\M{F}^*_i \right)$. 
Using properties of the 
trace, 
\begin{equation}\label{eq:inside_out} 
\tr\left[ \Gpos(\xi_i,\bvec{w}) \right] = 
\tr \left[ \Gpr \right]-\tr \left[ \noise\M{S}^{-1}_i(\bvec{w})\M{W}\M{F}_i\Gprsq\M{F}^*_i \right].
\end{equation}  
Denoting 
\[\M{K}(\xi_i,\bvec{w}) =  \noise\M{S}^{-1}_i(\bvec{w})\M{W}\M{F}_i\Gprsq\M{F}^*_i,
\]
the discretized OED objective $\bar{\phi}_N(\bvec w)$ in \eqref{eq:finite_expectation} can thus be rewritten as 
\begin{equation}\label{eq:phiNintro}
\bar{\phi}_N(\bvec w) = \tr \left[ \Gpr \right] + \phi_N(\bvec w), 
\end{equation} where   
\begin{equation}\label{eq:phiN}
\phiN(\bvec{w}) := -\frac{1}{N}\sum_{i=1}^{N} \tr \left[ \M{K}(\xi_i,\bvec{w}) \right]
=
-\frac{1}{N}\sum_{i=1}^{N}\sum_{j=1}^{d} \langle \bvec{e}_j , \M{K}(\xi_i,\bvec{w}) \bvec{e}_j \rangle. 
\end{equation} Here $\langle \cdot\,,\cdot \rangle$ denotes the Euclidean inner product.

Since $\Gpr$ is independent of
$\bvec{w}$, we can neglect that term in the optimization and focus on minimizing $\phi_N(\bvec w)$. The 
optimization problem for finding an OED is then,
\begin{equation}\label{eq:disc_functional}
\min_{\bvec{w} \in [0, 1]^{\Ns}}
{\phiN}(\bvec{w}) + \gamma \psi \left( \bvec w \right).
\end{equation}
Note that since $\bar{\phi}_N$ (and hence $\phi_N$)  
is convex, the above objective is convex as long as the
penalty function is convex. 

We next derive the gradient of $\phiN(\bvec{w})$.
We begin by defining some notations to facilitate this derivation.
Let us denote 
$\M{W}^\sigma := \noise \M{W}$. Note that 
\[
    \M{W}^\sigma =  \sum_{l = 1}^{\Ns} w_l \M{E}_l \quad
    \text{with} \quad
    \M{E}_l = \noise \M{I}_r \otimes \bvec{e}_l \bvec{e}_l^T,
\]
where 
$\bvec{e}_l$ the
$l$th coordinate vector in $\R^{\Ns}$ and $\M{I}_r$ the $r 
\times r$ identity matrix. 
%
%
%
%
%
Next, we consider the partial derivatives of $\M{S}_i(\bvec{w})$ 
which appears in the definition of
$\M{K}(\xi_i, \bvec{w})$ in~\eqref{eq:phiN}. Note that with the notation we 
just introduced 
\[
\M{K}(\xi_i,\bvec{w}) = \M{S}^{-1}_i(\bvec{w})\M{W}^\sigma\M{F}_i\Gprsq\M{F}^*_i, 
\quad\text{with}\quad
\M{S}_i(\bvec{w}) = \left( \M{I} + \M{W}^\sigma\M{F}_i\Gpr\M{F}^*_i \right), 
\quad i = 1, \ldots, N.
\]
It is straightforward to see that
\begin{equation}\label{eq:partialSi}
\frac{\partial \M{S}_i(\bvec w)}{\partial  w_k} 
= \M{E}_k \M{F}_i \Gpr \M{F}_i^*,
\quad k = 1, \ldots, s, \, i = 1, \ldots, N.
\end{equation}
Then, 
for $k = 1, \ldots, s$,
\begin{align}\label{eq:gradfull}
     \frac{\partial \phiN(\bvec w)}{\partial w_k} 
     &=-\frac1N\sum_{i=1}^N \sum_{j=1}^{d} 
        \langle \bvec{e}_j, 
        \left( \M{S}_i^{-1}(\bvec w)\frac{\partial\M{W}^\sigma}{\partial  w_k} - 
            \M{S}_i^{-1}(\bvec w)\frac{\partial \M{S}_i(\bvec w)}{\partial  w_k}\M{S}_i^{-1}(\bvec w)\M{W}^\sigma
        \right) 
        \M{F}_i \Gprsq \M{F}_i^* \bvec{e}_j\rangle \nonumber \\
    &= -\frac1N\sum_{i=1}^N\sum_{j=1}^{d} 
        \langle \bvec{e}_j ,  
        \M{S}_i^{-1}(\bvec{w})\M{E}_k 
        \left( \M{I} -\M{{F}}_i \Gpr \M{{F}}_i^* \M{S}_i^{-1}(\bvec w)\M{W}^\sigma
        \right) 
        \M{{F}}_i {\Gprsq}\M{{F}}_i^* \bvec{e}_j \rangle \nonumber\\
    &= -\frac1N\sum_{i=1}^N \tr \left[ \M{S}_i^{-1}(\bvec w)\M{E}_k \M{A}_i(\bvec{w}) \M{B}_i \right] ,
     \end{align}
where $\M{A}_i(\bvec{w}) := \M{I} -\M{{F}}_i \Gpr\M{{F}}_i^* \M{S}_i^{-1}(\bvec w)\M{W}^\sigma$
and $\M{B}_i := \M{{F}}_i {\Gprsq}\M{{F}}_i^*$, $i=1, \ldots, N$. 

\section{Model reduction}\label{sec:mod_red}

Evaluating the objective function in \eqref{eq:disc_functional} and its partial
derivatives in \eqref{eq:gradfull} requires many discretized PDE solves. 
Specifically, computation of the trace in the OED objective function
\eqref{eq:disc_functional} requires $d$ applications of $\M{K}(\xi_i,\bvec{w})$
for each $i = 1,\ldots,N$; see~\eqref{eq:phiN}. Each application of $\M{K}(\xi_i,\bvec{w})$ involves
a forward and adjoint PDE solve as well as the inverse of the operator
$\M{S}_i(\bvec{w})$, and each application of 
$\M{S}_i(\bvec{w})$ requires two PDE solves. 
Additionally, evaluating the partial derivatives $\frac{\partial
\phiN(\bvec w)}{\partial w_k} $ for each $k = 1,\ldots,s$ requires $d$
applications of $\M{S}_i^{-1}(\bvec w)\M{E}_k \M{A}_i(\bvec{w}) \M{B}_i$ for $i
= 1,\ldots,N$, each of which require four PDE solves and two applications of
$\M{S}_i^{-1}(\bvec w)$. Since computing the objective and its gradient is
required in each step of an optimization procedure, the resulting
large number of  PDE
solves can become computationally infeasible. 

In this section, we propose a method for replacing PDE solves in the OED
objective and gradient computation with a reduced model to make the optimization computationally
tractable.
The upfront computation of this reduced order model (ROM) for $\M{F}_i$
for $i = 1,\ldots,N$ requires upfront forward and adjoint PDE solves.
Once the reduced order models are computed, we can solve the
optimization problem for all choices of $\gamma$ without requiring
additional PDE solves. In section~\ref{subsec:comp_lr}, we exploit the
problem 
structure to find low-dimensional subspaces of the observation and parameter
spaces that capture the effective action for each forward operator sample
$\M{F}_i$, $i = 1, \ldots, N$ (preconditioned by the prior).  As discussed in
section~\ref{subsec:clustering}, this can be made more efficient by clustering
the samples of uncertain parameters $\{\xi_i\}_{i=1}^N$ 
such that the corresponding forward operator
samples ($\M{F}_i$'s) in each cluster share similar features.
Then, low-dimensional bases are computed for each cluster. 

\subsection{Composite low-rank basis}\label{subsec:comp_lr}
The article \cite{alexanderian:oed}, which concerns OED with no model
uncertainty, proposes computing a low-rank approximation to the
prior-preconditioned forward map $\widetilde{\M{F}}_i := {\M{F}}_i\Gprsqrt$ in
terms of a low-rank singular value decomposition (SVD); in that article only one copy
of $\widetilde{\M{F}}_i$ is considered, due to lack of model uncertainty. 
The prior-preconditioned forward operator is commonly low-rank due to
properties of the inverse problem, the limited number of observations
and the smoothing properties of the prior.  Following this 
procedure directly would require computing and storing the left and right singular vectors for each $\widetilde{\M{F}}_i$, $i = 1,\ldots,N$. While each individual
map may require a small number of vectors to approximate its effective domain and
range, computing and storing such a low-rank approximation for every
$\widetilde{\M{F}}_i$ individually 
could become
infeasible. Additionally, there may be some overlap in the singular vectors required to approximate each forward operator. Thus, we propose a method for finding spaces that capture the effective composite action of $\widetilde{\M{F}}_i$, $i = 1, \ldots, N$. 

Accordingly,
we seek to find two matrices, $\M{Q} \in \mathbb{R}^{d \times k}$ and $\widehat{\M{Q}} \in \mathbb{R}^{n \times k}$, with $k \in \mathbb{N}$ as small as possible, such that
\begin{equation}
\begin{aligned}\label{eq:composite_basis}
[\widetilde{\M{F}}_1,\ldots,\widetilde{\M{F}}_N] &\approx  \M{Q}{\M{Q}}^T [\widetilde{\M{F}}_1,\ldots,\widetilde{\M{F}}_N],  \\
[\widetilde{\M{F}}^*_1,\ldots,\widetilde{\M{F}}^*_N] &\approx \widehat{\M{Q}}\widehat{{\M{Q}}}^T[\widetilde{\M{F}}^*_1,\ldots,\widetilde{\M{F}}^*_N]. 
\end{aligned}
\end{equation}
 While the composite spaces spanned by $\M{Q}$ and $\widehat{\M{Q}}$ do not give us precise bounds for $\|\widetilde{\M{F}}_i - \M{Q}{\M{Q}}^T\widetilde{\M{F}}_i \widehat{\M{Q}} {\widehat{\M{Q}}^T} \|$ for each individual $\widetilde{\M{F}}_i$ ($i = 1,\ldots,N$), in practice we find that 
\begin{equation}\label{eq:FQF}
\widetilde{\M{F}}_i \approx \M{Q}{\M{Q}}^T\widetilde{\M{F}}_i \widehat{\M{Q}} {\widehat{\M{Q}}^T},
\quad i = 1,\ldots,N.
\end{equation}
We find $\M{Q}$ and $\widehat{\M{Q}}$ via the randomized range finder algorithm (RRF) 
\cite[Algorithm 4.1]{HalkoMartinssonTropp11}. The idea is to simultaneously compute a basis for the subspaces of $\mathbb{R}^d$ and $\mathbb{R}^n$ that capture the action of each $\widetilde{\M{F}}_i$  and $\widetilde{\M{F}}_i^*$ ($i = 1,\ldots,N$) respectively. To do so, we choose a $d \times r$ Gaussian random matrix\footnote{By Gaussian random matrix, $\M{\Omega} \in \mathbb{R}^{d \times k}$, 
we mean that entries of $\M{\Omega}$ are independent standard normal random variables.} $\M{\Omega}_i$ and an $n \times r$ Gaussian random matrix $\widehat{\M{\Omega}}_i$ and compute $\M{Y}_i = \widetilde{\M{F}}_i \M{\Omega}_i$ and  $\widehat{\M{Y}}_i = \widetilde{\M{F}}^*_i \widehat{\M{\Omega}}_i$ for each $i$. The SVDs of $[\M{Y}_1,\ldots,\M{Y}_N]$ and $[\widehat{\M{Y}}_1,\ldots,\widehat{\M{Y}}_N]$ are computed and truncated up to a specified tolerance to obtain $\M{Q}$ and $\widehat{\M{Q}}$ respectively. The algorithm is summarized in Algorithm \ref{alg:RRF}.

\begin{algorithm}[H]
\begin{algorithmic}[1]
\Procedure{CRF}{$[\widetilde{\M{F}}_1,\ldots,\widetilde{\M{F}}_N], [\widetilde{\M{F}}^*_1,\ldots,\widetilde{\M{F}}^*_N], \mu$}

 \State Given $N$ Gaussian random matrices $\M{\Omega}_i \in \mathbb{R}^{d\times r}$ and $N$ Gaussian random matrices $\widehat{\M{\Omega}}_i \in \mathbb{R}^{n \times r}$, $i = 1,\ldots,N$
 \State Compute $\M{Y}_i = \widetilde{\M{F}}_i\M{\Omega}_i$ and $\widehat{\M{Y}}_i = \widetilde{\M{F}}^*_i\widehat{\M{\Omega}}_i$
 \State Set $\M{Y} = [\M{Y}_1,\ldots,\M{Y}_N]$ and $\widehat{\M{Y}} = [\widehat{\M{Y}}_1,\ldots,\widehat{\M{Y}}_N]$
 \State Compute SVDs of $\M{Y} = \M{U}\M{\Sigma}\M{V}^T$ and of $\widehat{\M{Y}} = \widehat{\M{U}}\widehat{\M{\Sigma}}\widehat{\M{V}}^T$ with singular values $\sigma_j$ and $\widehat{\sigma}_l$ ($j = 1,\ldots,d$, $l = 1,\ldots,n$) respectively 
 \State Set $k = \max \{ \max \left({j} \mbox{  s.t.  } \frac{\sigma_{j}}{\sigma_{1}} \leq \mu \right), \max \left({l} \mbox{  s.t.  } \frac{\widehat{\sigma}_{l}}{\widehat{\sigma}_{1}} \leq \mu \right)  \}$
 \State Set $\M{Q} $ to be the first $k$ columns of $\M{U}$ and $\widehat{\M{Q}}$ to be the first $k$ columns of $\widehat{\M{U}}$ \\
 \Return $\M{Q}$ and $\widehat{\M{Q}}$
 \EndProcedure
 \end{algorithmic}
 \caption{Composite randomized range finder algorithm}\label{alg:RRF} 
\end{algorithm}

To fully eliminate the PDEs from the optimization problem, the small
inner matrices ${\M{Q}}^T\widetilde{\M{F}}_i\widehat{\M{Q}}$ in
\eqref{eq:FQF} must be computed and stored for each $i = 1,\ldots,N$,
which requires solving $k$ more PDEs for each forward operator sample. If
desired, these additional PDE solves can be avoided at the cost of
additional error by modifying the single-pass approach presented in
\cite{HalkoMartinssonTropp11}. Specifically, a matrix $\M{B}_i$
approximating ${\M{Q}}^T\widetilde{\M{F}}_i\widehat{\M{Q}}$ for each
$i = 1,\ldots,N$ can be found using a minimal residual method to approximately satisfy the relations
\begin{equation}\label{eq:single-pass}
\M{B}_i\widehat{\M{Q}}^T\M{\Omega}_i = \M{Q}^T \M{Y}_i \quad \text{and} \quad
\M{B}_i^T\M{Q}^T\widehat{\M{\Omega}}_i = \widehat{\M{Q}}^T \widehat{\M{Y}}_i. 
\end{equation}

\subsection{Clustering}\label{subsec:clustering}
To reduce the amount of PDE solves needed to compute the inner
matrices and the amount of basis vectors stored, we follow the ideas
presented in \cite{peherstorfer2014} and break up the sample space
$\Omega$ into $l$ clusters. To do this, we use a standard $k$-means
clustering algorithm where we define the distance measure between two
samples $\xi_i$ and $\xi_j$ as the Euclidean distance between the
observations they produce for an instance of the inversion parameter
$\bvec m$. Specifically,
for an $\bvec m \in \mathbb{R}^n$, we define the distance between two
samples as
\begin{equation}\label{eq:distance} d(\xi_i,\xi_j;\bvec{m})
  = \| \M{F}_i\bvec{m}-\M{F}_j\bvec{m} \|_2.
\end{equation} Naturally, this distance depends on the choice of
$\bvec{m}$. One could pick $\bvec{m}$ as a random draw from the prior
distribution or one could pick a suitable parameter $\bvec{m}$, which is an
interesting design problem itself.  For the model problem used in this
paper, we choose the parameter as a sum of
radial basis functions with different centers and
magnitudes.

We compute matrices $\M{Q}_{p}$ and $\widehat{\M{Q}}_p$ of rank $k_p$ using
Algorithm \ref{alg:RRF} for each cluster $p = 1,\ldots,l$ to approximate the
effective action of all the forward operators in the cluster. This preliminary
clustering step allows us to reduce the amount of PDE solves needed when
computing the inner matrices in Algorithm \ref{alg:RRF}. In addition to $k_p$
being smaller than $k$ due to the clusters containing less samples, sample
parameters $\xi_i$ and $\xi_j$ corresponding to the same cluster $p$ produce
more similar data than parameters in different clusters. This means that
forward operators corresponding to the same cluster have overlapping range
spaces, thus requiring less basis vectors to cover them. 

The updated composite low-rank basis algorithm given is provided in Algorithm
\ref{alg:composite_basis}, where the $\text{CLUSTER}$ procedure refers to a standard k-means clustering procedure which takes as inputs the synthetic observations obtained for an instance of the inversion parameter ($[\widetilde{\M{F}}_1\bvec m,\ldots,\widetilde{\M{F}}_N\bvec m]$), the number of desired clusters $l$, and a tolerance tol, and outputs the clusters $[c_1,\ldots,c_N]$ for each forward operator $\widetilde{\M{F}}_i$ ($i = 1,\ldots,N$).

\begin{algorithm}[H]
\begin{algorithmic}[1]
\Procedure{CRRFWC}{$[\widetilde{\M{F}}_1,\ldots,\widetilde{\M{F}}_N],[\widetilde{\M{F}}^*_1,\ldots,\widetilde{\M{F}}^*_N],\mu, l,\bvec{m}$}
 \State $[c_1,\ldots,c_N] = \mbox{CLUSTER}([\widetilde{\M{F}}_1\bvec m,\ldots,\widetilde{\M{F}}_N\bvec m], l, \mbox{tol})$
 \State Initialize $\M{Q} =$ [ ] and $\widehat{\M{Q}} =$ [ ]
\For{$i = 1,\ldots,l$}
 \State Initialize $\widetilde{\M{F}} =$ [ ] and $\widetilde{\M{F}}^* = $ [], 
 \For{$j = 1,\ldots,N$}
   \If{$c_j == i$}
   append $\widetilde{\M{F}}_j$ to $\widetilde{\M{F}}$ and append $\widetilde{\M{F}}^*_j$ to $\widetilde{\M{F}}^*$ 
 \EndIf
 \EndFor
 \State Compute $[\M{Q}_{i},\widehat{\M{Q}}_{i}] = \mbox{CRF}[\widetilde{\M{F}},\widetilde{\M{F}}^*,\mu]$
 \State Append to $\M{Q}$, $\M{Q}  = [\M{Q} ,\M{Q}_{i}]$ and to $\widehat{\M{Q}}$, $\widehat{\M{Q}}  = [\widehat{\M{Q}} , \widehat{\M{Q}}_{i}]$
\EndFor
\For{$i = 1,\ldots,N$}
\State $\M{Q}_{c_i}\M{F}\widehat{\M{Q}}_{c_i} = \M{Q}_{c_i}\widetilde{\M{F}}_{i}\widehat{\M{Q}}_{c_i}$
\EndFor
\Return [$\M{Q}_1,\ldots,\M{Q}_l$],[$\widehat{\M{Q}}_1,\ldots,\widehat{\M{Q}}_l$] and [$\M{Q}_{c_1}\M{F}\widehat{\M{Q}}_{c_1},\ldots,\M{Q}_{c_N}\M{F}\widehat{\M{Q}}_{c_N}$]
 \EndProcedure
 \end{algorithmic}
 \caption{Composite randomized range finder algorithm with clustering
}\label{alg:composite_basis}
\end{algorithm}


\section{Subsurface flow example: the inverse problem}\label{sec:IP}

This section is devoted to the description of the example inverse
problem used to illustrate our methods.  The subsurface flow forward
and inverse problems are described in this section. The sources of
irreducible uncertainty, which enter in this problem, are discussed in
section~\ref{sec:darcy}, and optimal designs taking into account the
irreducible uncertainty are the topic of section~\ref{sec:numerics}.
The inverse problem we consider seeks to estimate an 
uncertain initial contaminant concentration field using sensor
measurements of contaminant concentration recorded at a discrete set
of observation times.

\subsection{The forward problem}\label{subsec:motivation}
We consider the transport of a contaminant in a rectangular domain $D := [0,a]
\times [0,b] \subset \mathbb{R}^2$.~\footnote{This two-dimensional setting can be
understood as a top-down view of the evolution of an initial concentration in a
horizontal slice of an aquifer or a slice resulting from averaging the
properties of a thin 3-dimensional domain.}
The evolution of the contaminant's concentration, $u(\bvec{x}, t)$, in 
groundwater flow is modeled
by a time-dependent advection-diffusion equation
\begin{equation}\label{eq:ad-diff}
\begin{aligned}
u_t-\kappa\Delta{u}+\bvecS v\cdot\nabla{u} &= 0 \hspace{3mm} &&\mbox{   in } D \times (T_0,T_1), \\
u(\cdot,T_0) &= m \hspace{3mm} &&\mbox{   in } D, \\
(-\kappa\nabla{u} + u \bvecS{v}) \cdot\bvecS n  &= 0 \hspace{3mm} &&\mbox{   in }
\Gamma_{\!l} \times (T_0,T_1), \\
\kappa\nabla{u}\cdot \bvecS n &= 0 \hspace{3mm} &&\mbox{   in }
\partial D \setminus \Gamma_{\!l} \times (T_0,T_1),
\end{aligned}\end{equation}
where $\kappa > 0$ is a known diffusion coefficient, $\bvecS{v}$ is the
advection velocity field, $T_0 < T_1$ are the initial and the final
time, respectively,
and $m \in L_2(D)$ is the initial concentration field. 
Here, $\Gamma_{\! l}:=\{0\}\times [0,b]$ is the left boundary of the domain.
We assume $\Gamma_{\! l}$ is impermeable, as modeled by the zero  
total flux condition at that boundary. We want the contaminant to be able to leave the domain, so we allow it to 
advect freely through the remaining portion of the boundary, $\partial
D \setminus \Gamma_{\! l}$; this is modeled 
by imposing a homogeneous pure Neumann condition.
For a more detailed explanation and a derivation of a model for two-dimensional
flow in an aquifer, see e.g., \cite[section 5.3]{bear2018modeling}.  

A major source of uncertainty in the governing
equation~\eqref{eq:ad-diff} is the velocity field $\bvecS{v}$, which is
an irreducible model uncertainty considered herein. Moreover, the
time interval $[T_0,T_1]$ might be uncertain as we can only estimate
how long ago a contaminant has been released.  Our model for
these irreducible uncertainties is detailed in section~\ref{sec:darcy}.
Before that, we detail the Bayesian inverse problem for fixed
advection velocity $\bvecS{v}$ and time interval $[T_0,T_1]$.

\subsection{Bayesian inversion for initial state}\label{subsec:subsurf_fm}
For the Bayesian inversion of the initial concentration $m$ given a fixed 
velocity field $\bvecS{v}_i$ and a fixed time interval $[T_0,T_1]$, 
we impose a Laplacian-like prior. 
We define the prior operator on the domain $D$ as $\Cpr :=
\mathcal{A}^{-2} := (-\rho\Delta+\delta I)^{-2}$. To reduce the variance
near the boundary of the domain resulting from combining the
differential operator $\mathcal{A}$ with Neumann boundary conditions,
we instead impose
Robin boundary conditions \cite{daon2016mitigating,roininen2014whittle}. 
Accordingly, given $s \in L^2(D)$, the weak solution $m$ of $\mathcal{A}m = s$ 
satisfies
\begin{equation}\label{eq:priorWeak}
\rho \int_D \nabla m \cdot \nabla p \, d\bvec{x} + \delta \int_D m p \,d\bvec{x} + \beta \int_{\partial D} mp \, ds = \int_{D} sp \,d\bvec{x}, \quad \text{for all } p \in H^1(D),
\end{equation}
where  $\beta = ({\rho}/{1.42}) \sqrt{{\delta}/{\rho}}$ as
proposed in \cite{roininen2014whittle}. The parameters $\rho$ and
$\delta$  control the correlation length and variance of the
covariance operator, and were set to $\rho = 0.008$ and $\delta =
0.02$ for our simulations as these parameters lead to prior samples
which have realistic smoothness and correlation lenght.

We choose $\Ns = 234$ candidate locations (these locations are shown
in the top right graphic in Figure~\ref{fig:MAPs}) where we can place
sensors. Setting $T_1 = 16$, we assume that at each candidate
location, we can take $\Nt = 5$ concentration measurements at equally
spaced observation times $\tau_i = 7,9,11,13,15$.
Measurements are time-averaged
concentrations over intervals $[\tau_i-0.5,\tau_i+0.5]$ for each $i = 1,\ldots,5$. This leads
to a vector of possible (i.e., if sensors were to be placed on all possible
locations) observations $\bvec{d} \in \mathbb{R}^{\Ns\Nt}$.

\begin{figure}
\centering
\begin{tikzpicture}
\node (11) at (-5.5,4)
  {\includegraphics[width=0.5\linewidth]{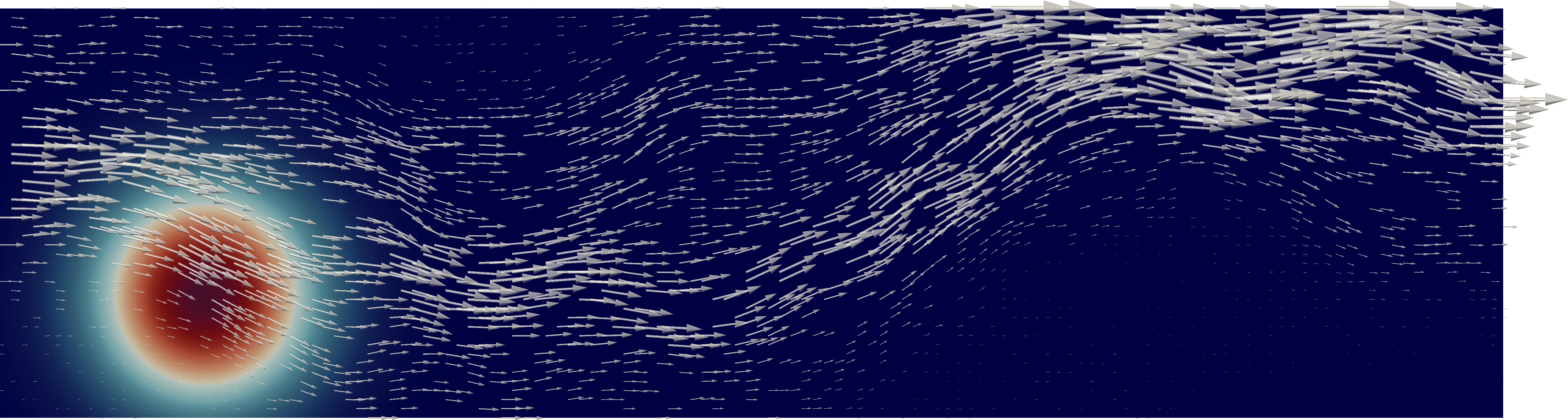}};
\node (12) at (2.5,4)
  {\includegraphics[width=0.5\linewidth]{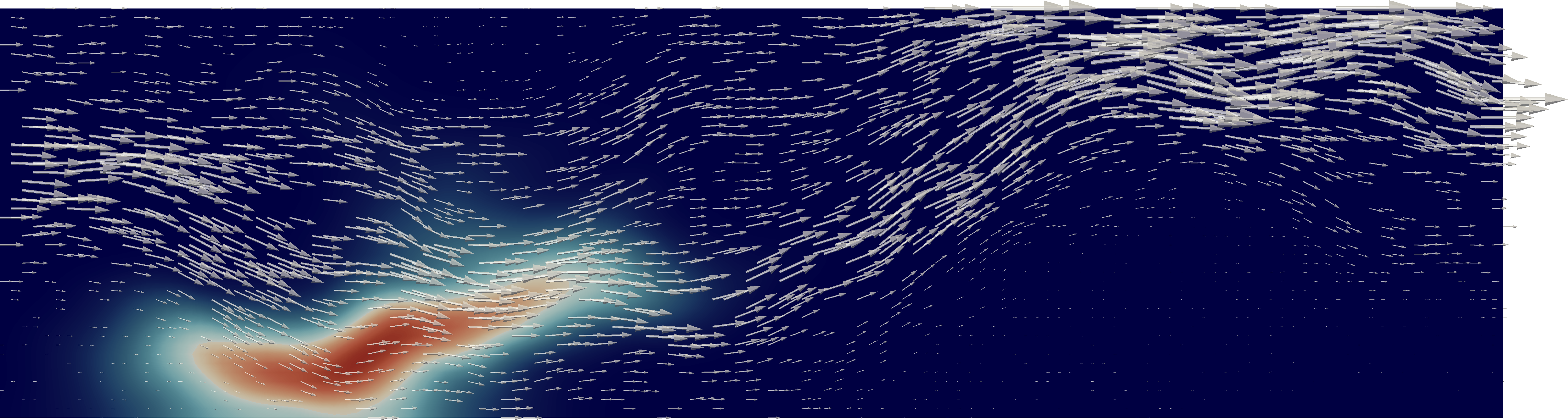}};
\node (21) at (-5.5,2)
  {\includegraphics[width=0.5\linewidth]{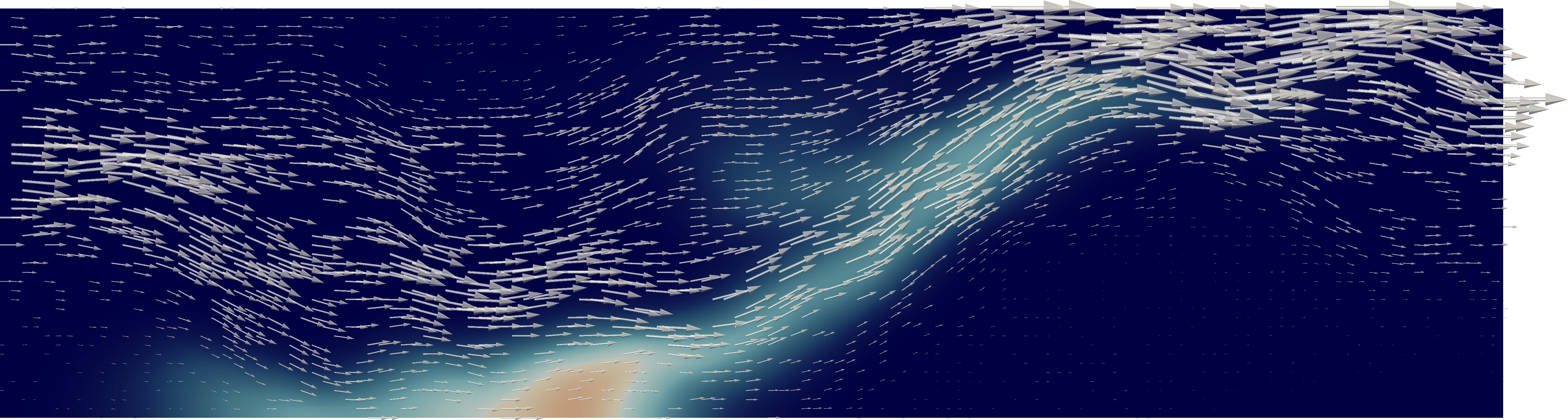}};
\node (22) at (2.5,2)
  {\includegraphics[width=0.5\linewidth]{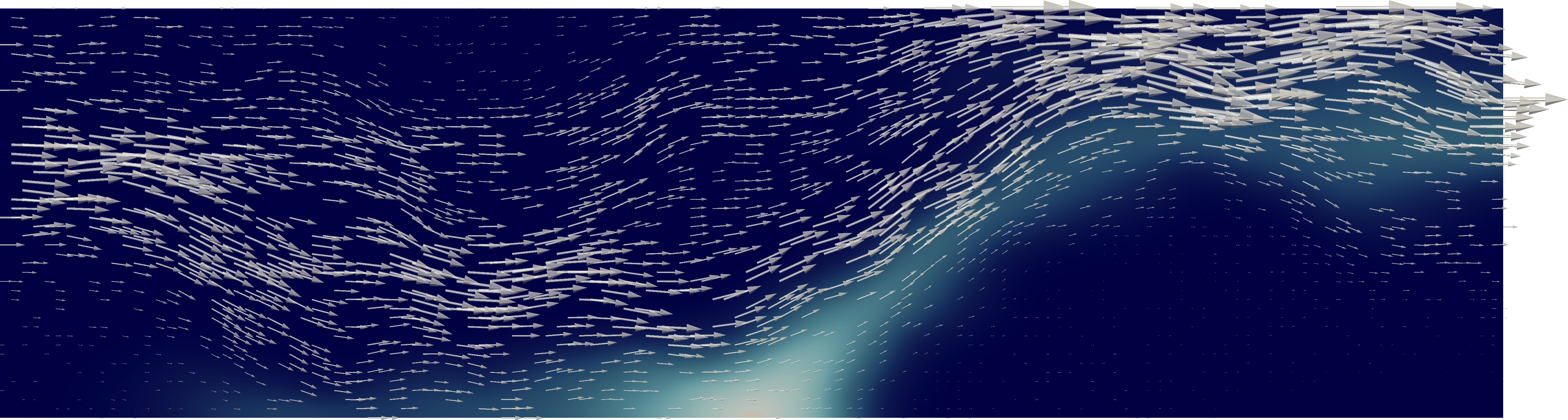}};
  \node at (-2.4,3.4) {\small{\color{white}{$t=0$}}};
  \node at (5.6,3.4) {\small{\color{white}{$t=3$}}};
  \node at (-2.4,1.25) {\small{\color{white}{$t=8$}}};
  \node at (5.6,1.25) {\small{\color{white}{$t=13$}}};
\end{tikzpicture}
\caption{Snapshots of the forward propagation of an initial concentration ($t=0$). The white arrows depict the velocity field used for the simulation.}
\label{fig:fwd_Prop}
\end{figure}
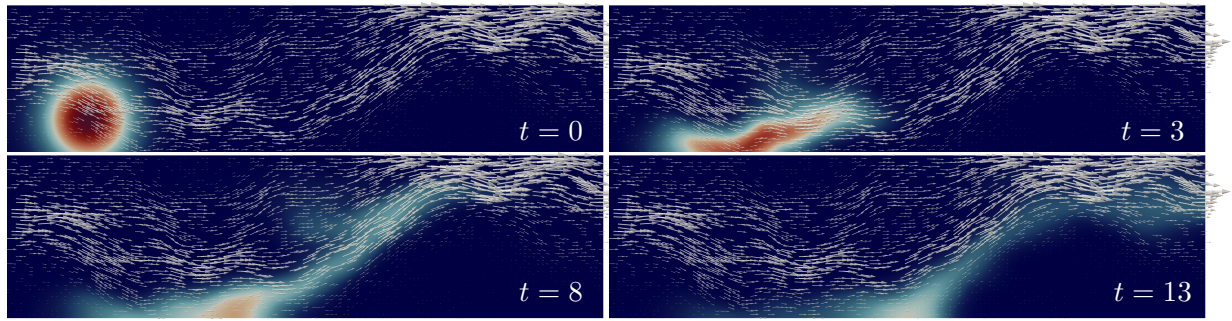

As discussed in section~\ref{subsec:design}, for a fixed velocity
field $\bvecS{v}$, time interval $[T_0,T_1]$, and design $\bvec{w}$,
under the assumption of Gaussian prior and additive Gaussian noise,
the posterior distribution is also Gaussian. In particular, it is
fully characterized by its mean $\mupost$, and posterior covariance
matrix, $\Cpost$. Different sensor placements lead to different MAP
points and different updates to 
the prior covariance matrix, as can be seen in Figures~\ref{fig:MAPs}
and \ref{fig:ptwise_Var}. The advection velocity field used for these
results is also shown in Figure~\ref{fig:MAPs}.
We illustrate the
effect of different design choices on the posterior pointwise variance
in Figure~\ref{fig:ptwise_Var} setting $T_0 = -1$ and $T_1 =
16$. In Figure~\ref{fig:MAPs} we show the MAP points ($\mupost$)
obtained for these same choices of design.

\begin{figure}
\centering
\begin{tikzpicture}
\node (11) at (-5.8,4)
  {\includegraphics[width=0.5\linewidth]{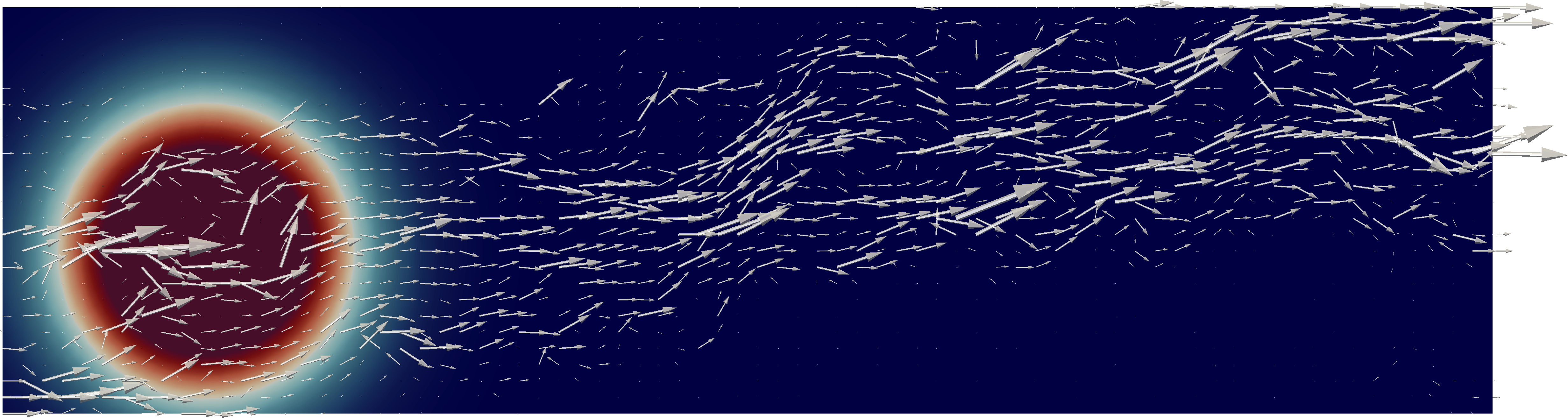}};
\node (12) at (2.0,4)
  {\includegraphics[width=0.47\linewidth]{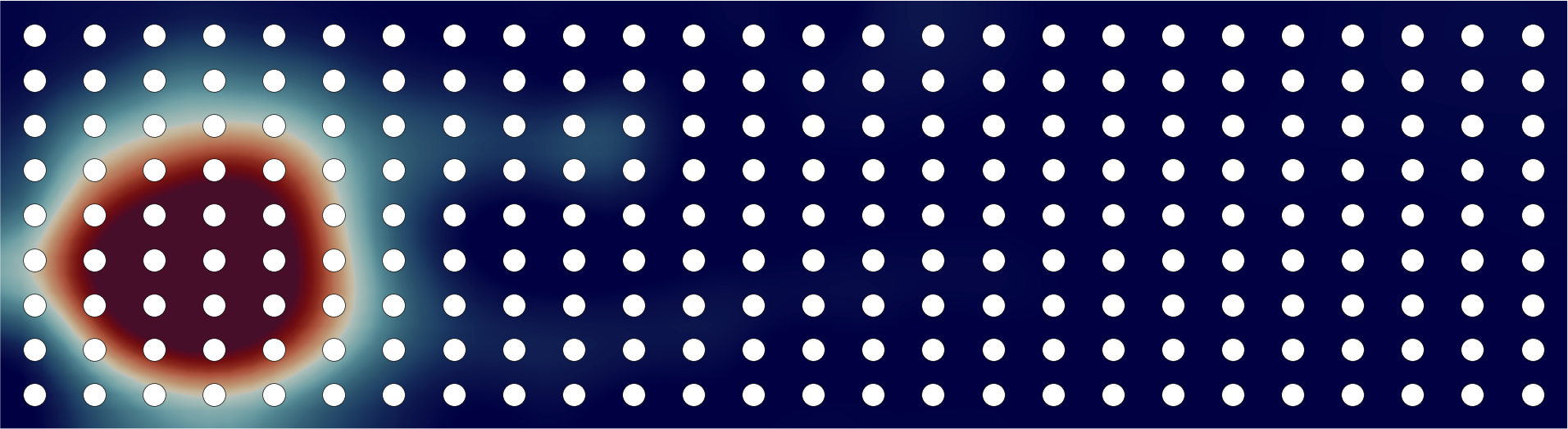}};
\node (21) at (-6.0,1.7)
  {\includegraphics[width=0.475\linewidth]{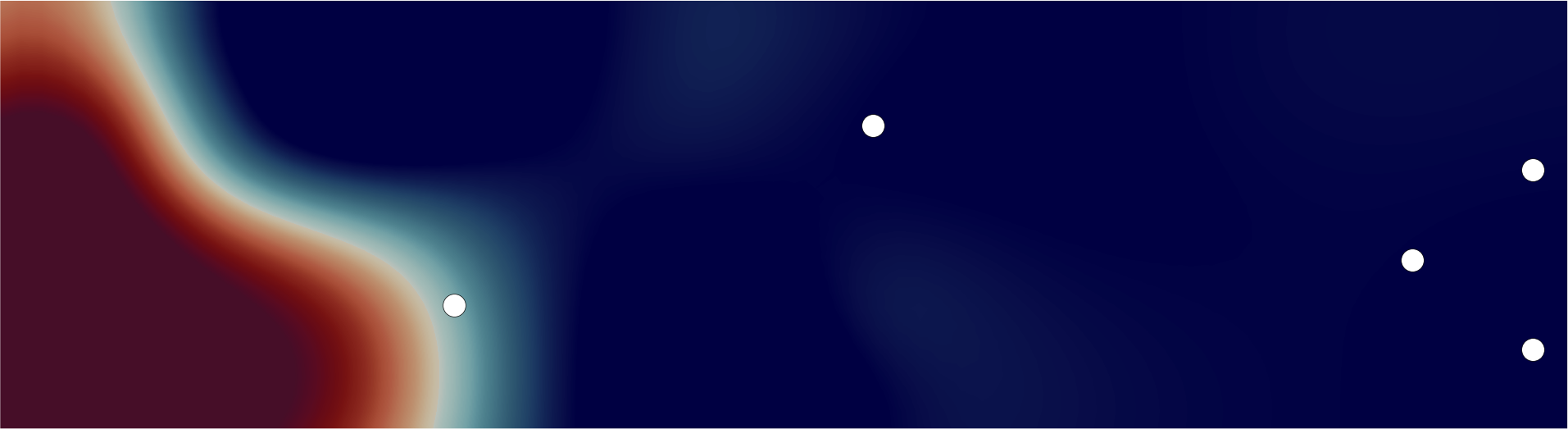}};
\node (22) at (2.0,1.7)
  {\includegraphics[width=0.47\linewidth]{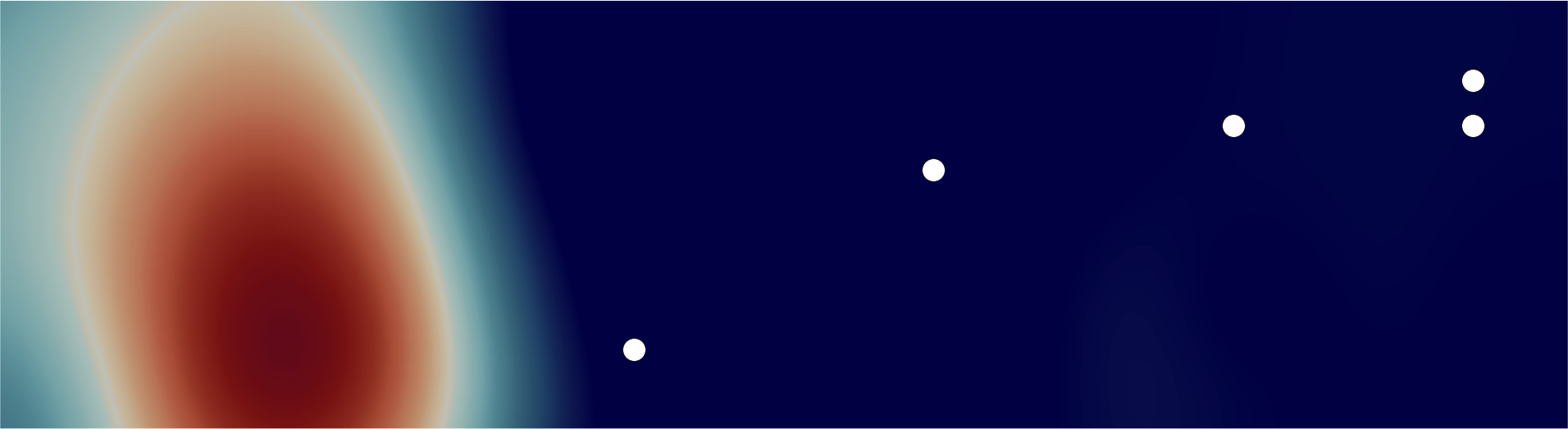}};
\node at (7.0,2.55) 
{\includegraphics[width=0.1\linewidth]{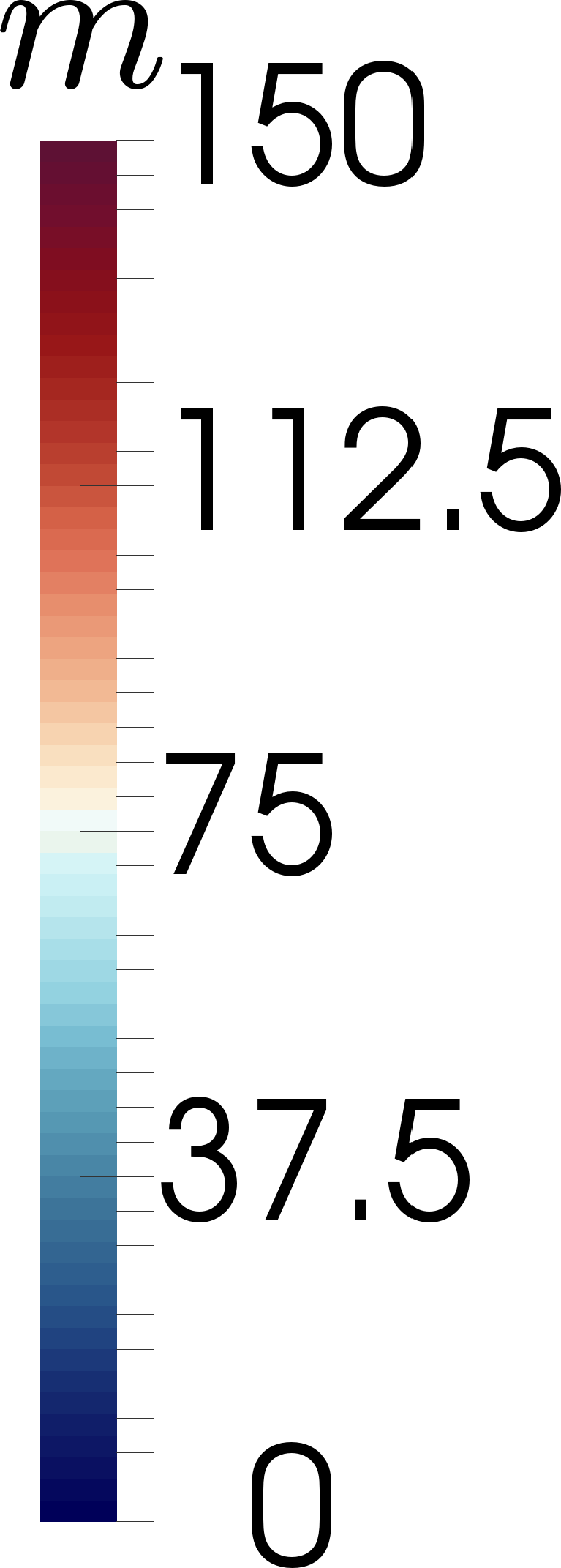}};
  \node at (-2.6,3.3) {\small{\color{white}{True}}};
  \node at (-2.4,0.9) {\small{\color{white}{D1}}};
  \node at (5.5,0.9) {\small{\color{white}{D2}}};
\end{tikzpicture}
\caption{Discrete MAP points, or $\mpost$, at the ``true''
  velocity field and time interval $[-1,16]$ obtained for different sensor locations. The top left
  figure shows the ``true'' initial concentration, $\bvec{m}$, we wish to
  characterize in \eqref{eq:ad-diff}, overlaid with the velocity
  field $\bvec{v}^*$. The top right figure shows the MAP point
  obtained using all 234 possible sensor and thus depicts the best we
  can do for our particular problem formulation. On the bottom, we
  show MAP points obtained with two different designs
  (D1 and D2) consisting of 8 sensors.}
\label{fig:MAPs}
\end{figure}
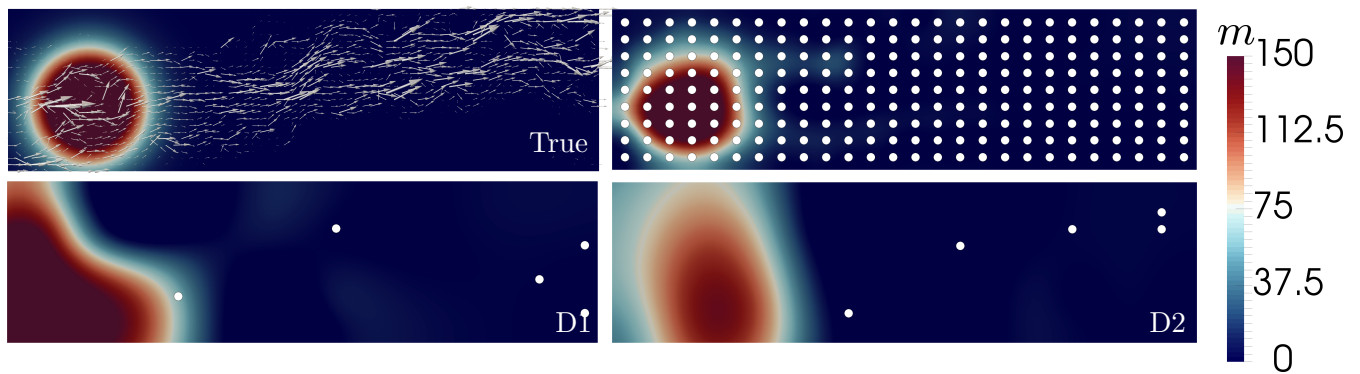

\begin{figure}
\centering
\begin{tikzpicture}
\node (11) at (-6.15,4)
  {\includegraphics[width=0.33\linewidth]{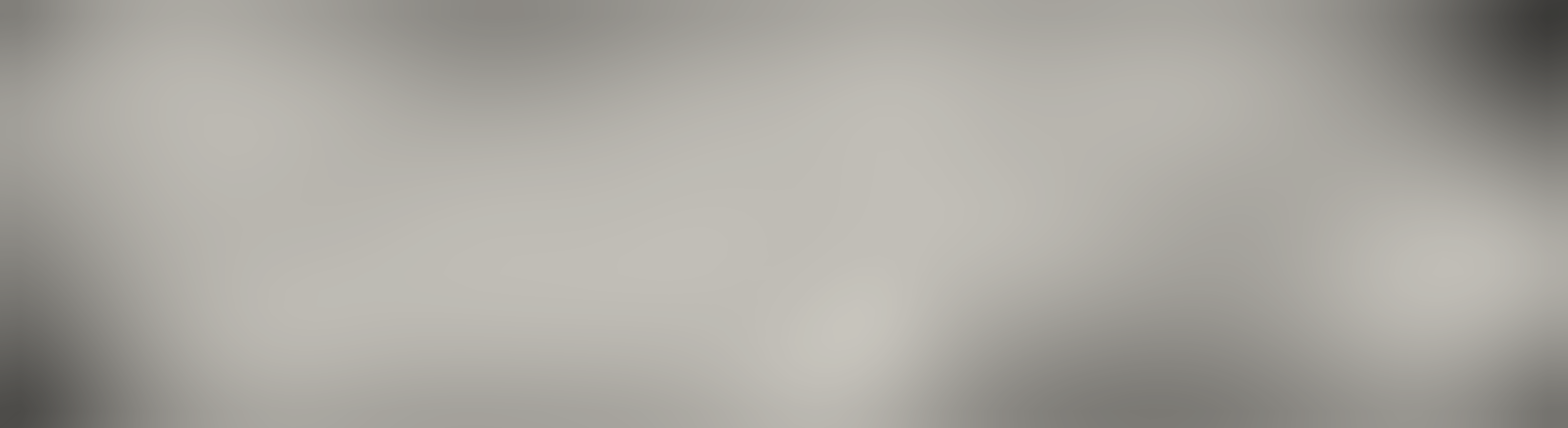}};
\node (12) at (-0.61,4)
  {\includegraphics[width=0.33\linewidth]{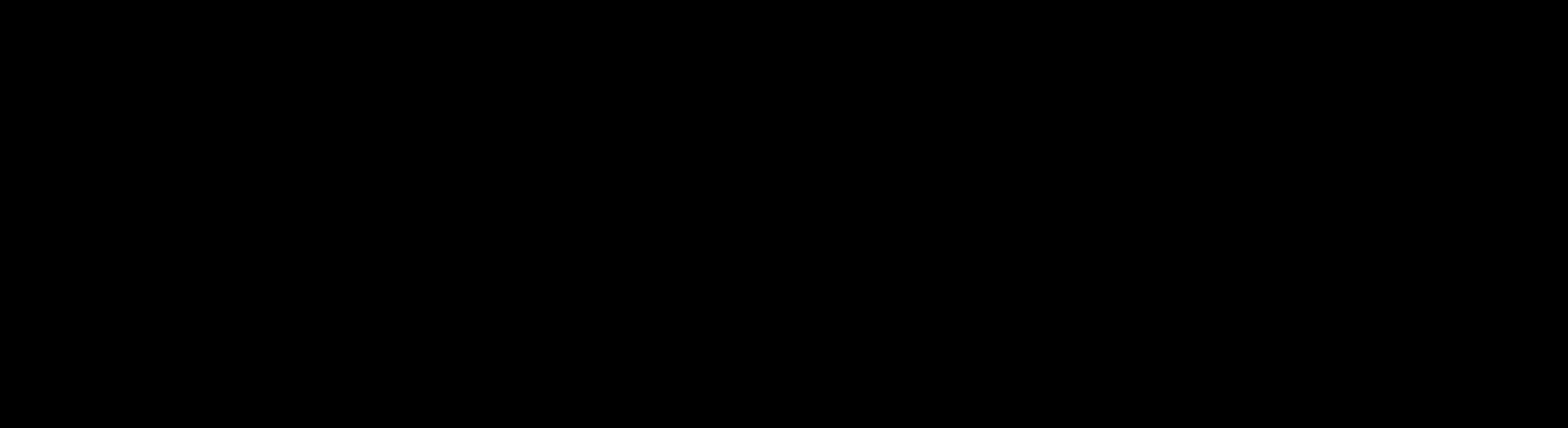}};
\node (21) at (4.91,4)
  {\includegraphics[width=0.33\linewidth]{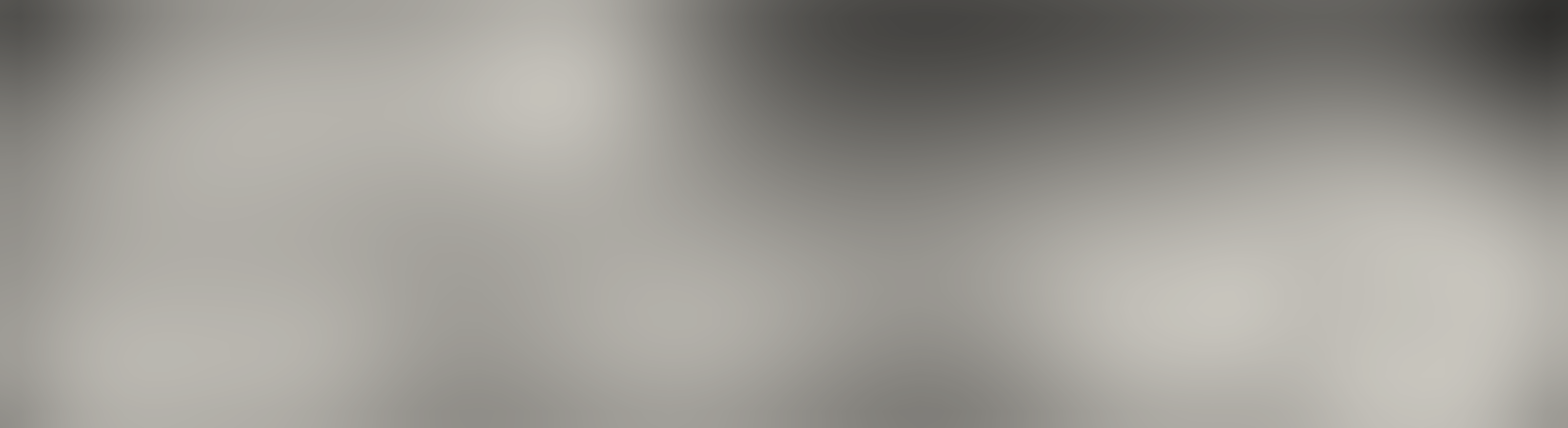}};
\node (22) at (8.1,4.2)
  {\includegraphics[width=0.1\linewidth]{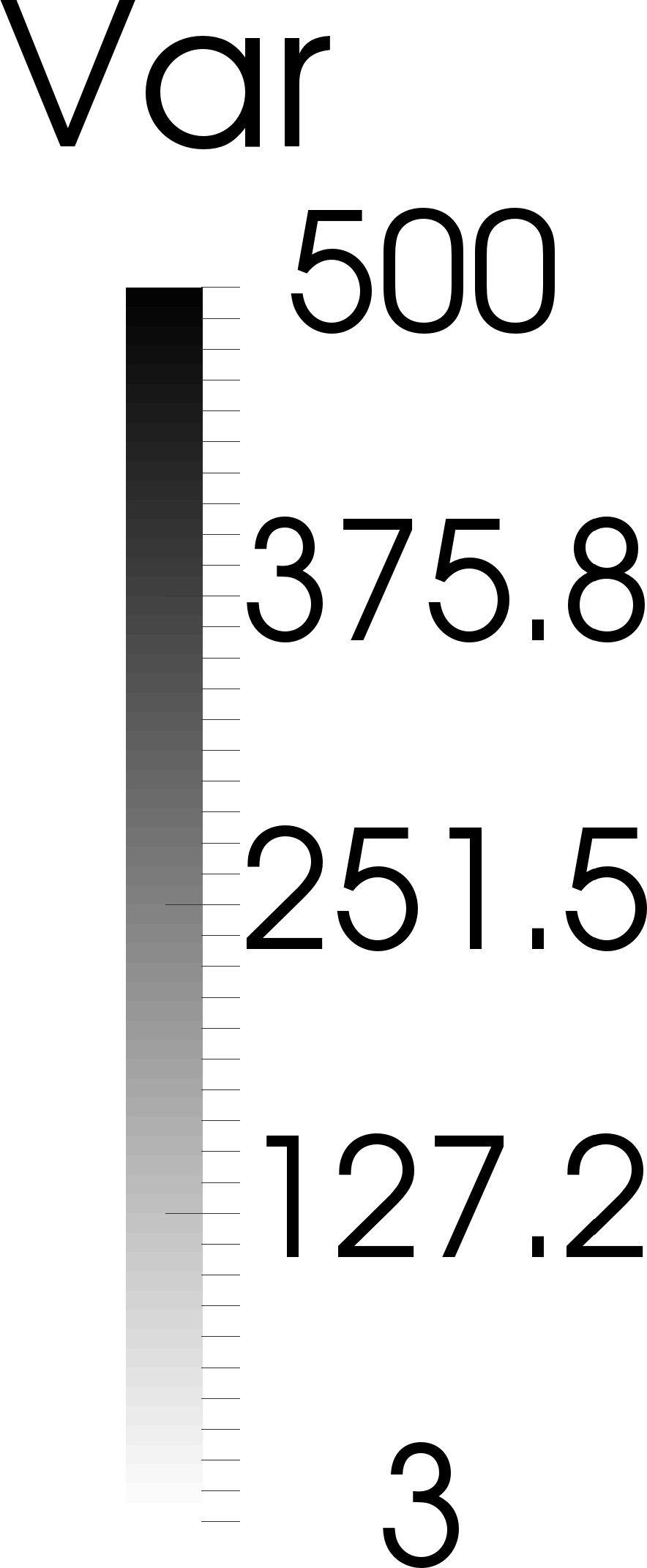}};
  \node at (-3.8,3.5) {\small{\color{black}{D1}}};
  \node at (1.4,3.5) {\small{\color{white}{Prior}}};
  \node at (6.8,3.5) {\small{\color{black}{D2}}};
\end{tikzpicture}
\caption{Shown is the pointwise variance for the prior (middle) and
  posterior using the designs D1 and D2 shown in the bottom row of Figure \ref{fig:MAPs}.}
\label{fig:ptwise_Var}
\end{figure}
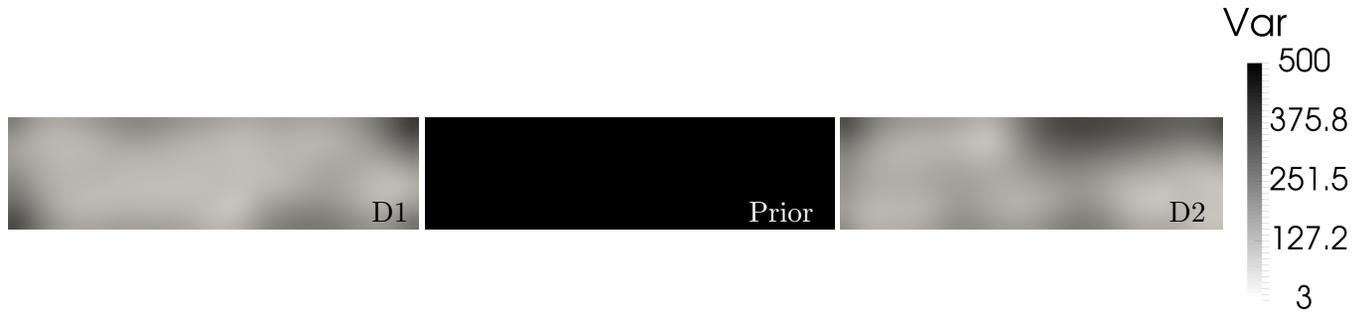

\subsection{Discretization and implementation}
To discretize the forward and adjoint operators as well as the prior
operator, we use the hIPPYlib~\cite{villa2019hippylib} framework. We
use implicit Euler for time-stepping using 250 timesteps in the interval $(T_0,T1]$. A streamline upwind Petrov
Galerkin (SUPG) discretization in space is used to stabilize the
discretization of the advective term. The spatial discretization uses 3750 triangles resulting in 1976 degrees of freedom.
To accelerate the linear solves
in the implicit Euler steps, we build sparse LU factorizations of the
spatial forward and adjoint operators and reuse them throughout the
implicit Euler iterations.  The evolution for one choice of velocity
field and initial condition is shown in Figure \ref{fig:fwd_Prop}.
All other solver components detailed in sections~\ref{sec:darcy} and
\ref{sec:numerics} also use hIPPYlib for the PDE discretization and
Python for the numerical linear algebra.

\section{Subsurface flow example: irreducible uncertainties}\label{sec:darcy}
In this section, we characterize the irreducible uncertainties we must take
into account when computing optimal sensor locations for the inverse problem
presented in section~\ref{subsec:subsurf_fm}. The sources of irreducible
uncertainty are the advection velocity $\bvecS{v}$ and the initial time $T_0$ in
\eqref{eq:ad-diff}. 
As explained further below,
the uncertainty in the velocity
field~\eqref{eq:velocity} stems from the 
log-permeability field $\theta(\bvec{x})$ 
of the aquifer
being uncertain. Therefore, the random variable $\xi$, 
introduced in section~\ref{subsec:bayes_uncertain}, that 
parameterizes model uncertainty is given by $\xi = (\theta, T_0)$.

\subsection{Characterizing the irreducible uncertainty in the
  subsurface flow velocity}
We will use a steady state inverse problem, different from
the inverse problem discussed in section~\ref{subsec:subsurf_fm}, which
is the main target of this work, to characterize a distribution of the
velocity field $\bvecS{v}$ and obtain samples from this distribution.
Darcy's law describes the flow of a fluid through a medium in terms of
the physical properties of the medium and the pressure gradient.
Using Darcy's law, the background velocity field $\bvecS{v}$
in~\eqref{eq:ad-diff} is described by
\begin{equation}\label{eq:velocity}
\bvecS{v}(\bvec{x}) = -e^{\theta(\bvec{x})}\nabla{p(\bvec{x})},
\end{equation}
where $\theta(\bvec{x})$ is the log-permeability field of the aquifer and
$p(\bvec{x})$ denotes the pressure of the groundwater transporting the
contaminant through the medium.

The equation governing the pressure field is obtained using Darcy's law along
with mass conservation and assuming incompressibility (of the fluid carrying
the contaminant).
This results in a linear elliptic PDE that can be written in the 
following dual-mixed form: 
\begin{equation}\label{eq:DarcyMixed}
\begin{aligned}
-\bvecS{v} - e^{\theta(\bvec{x})}\nabla{p} &= 0 \hspace{3mm} &&\mbox{   in } D, \\
\nabla \cdot \bvecS{v} &= 0 \hspace{3mm} &&\mbox{   in }D,\\
p &= p_0 \hspace{3mm} &&\mbox{   on } \Gamma_{\! l} \cup \Gamma_{\! r}, \\
-\bvecS{v} \cdot \bvecS{n} &= 0 \hspace{3mm} &&\mbox{   on }
\partial D \setminus (\Gamma_{\! l} \cup \Gamma_{\! r}). 
\end{aligned}\end{equation}
Here, $\Gamma_{\! l}:=\{0\}\times [0,b]$ and $\Gamma_{\!
  r}:=\{a\}\times [0,b]$ denote the left and right domain boundaries, respectively.
The Dirichlet boundary conditions are prescribed 
as $p_0 \equiv 0$ on $\Gamma_{\! l}$ and $p_0 \equiv 1$ on $\Gamma_{\!
  r}$.  
We mention that using the mixed form~\eqref{eq:DarcyMixed} when deriving the weak
formulation and finite element discretization  ensures mass conservation over
the elements in the numerical solution.  Also, porosity and fluid viscosity are
omitted from \eqref{eq:velocity} and \eqref{eq:DarcyMixed} because they are assumed
constant and thus can be absorbed in the remaining terms in the equations
through scaling.

As mentioned earlier, the uncertainty in the velocity
field~\eqref{eq:velocity} is due to uncertainty in 
the log-permeability field $\theta(\bvec{x})$.
One way to obtain a statistical distribution for the
uncertain permeability field is to solve a Bayesian inverse problem
governed by the forward problem described in \eqref{eq:DarcyMixed}.
This is described next.

\begin{figure}[ht]
\centering
\includegraphics[width=.45\textwidth]{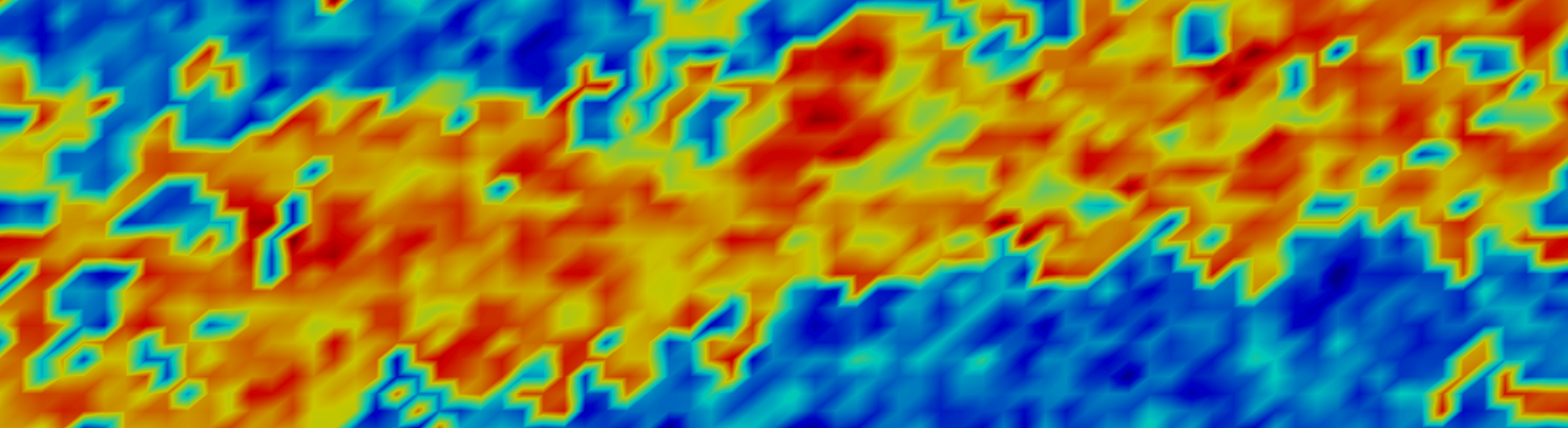}
\includegraphics[width=.45\textwidth]{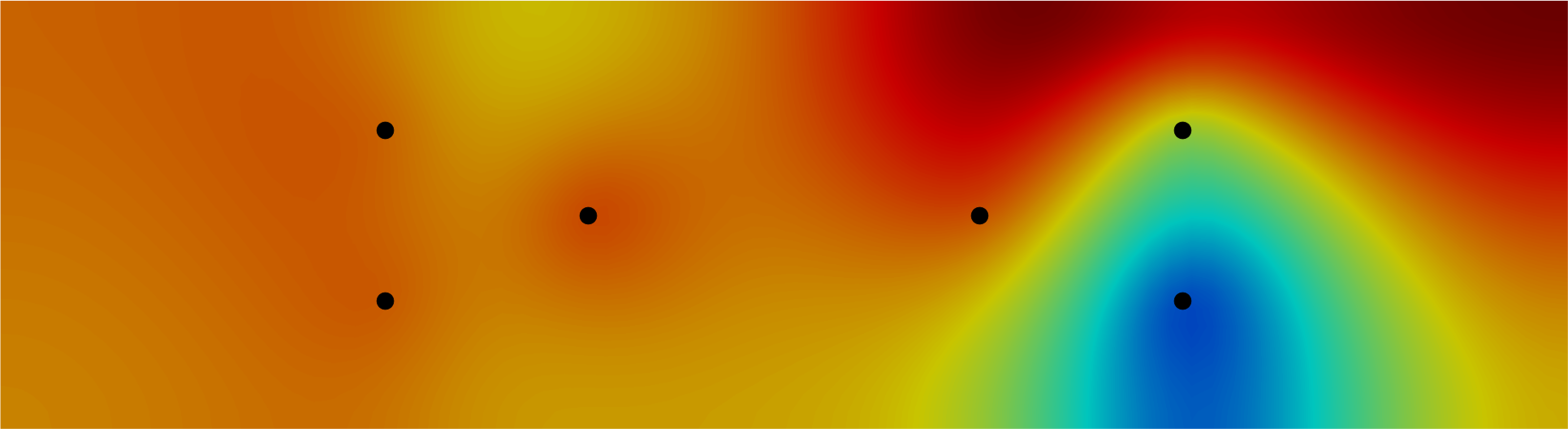}
\includegraphics[width=.08\textwidth]{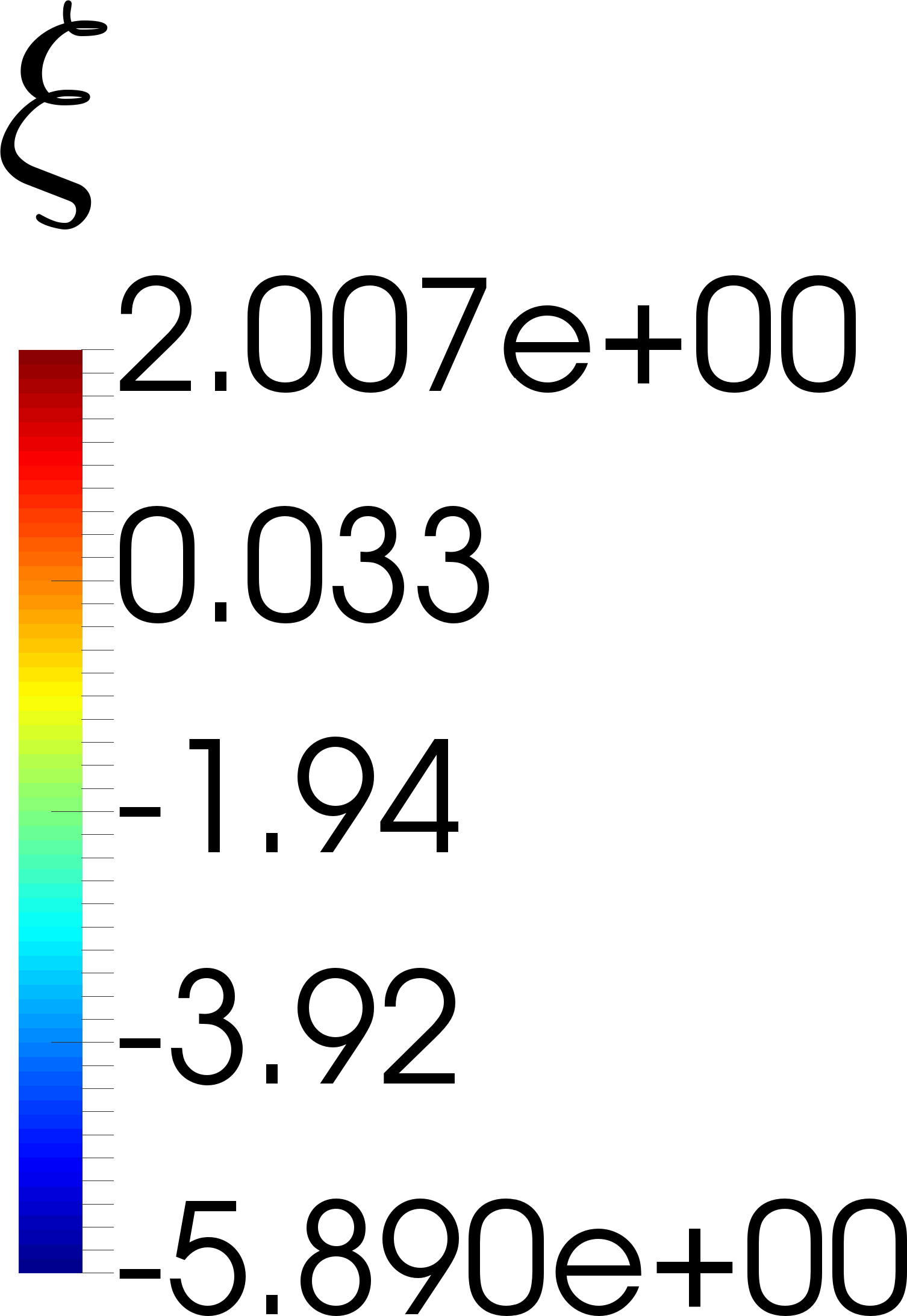}
\caption{Left: the ``true'' log permeability field used to generate synthetic pressure 
data to estimate the log-permeability field. 
This ``true'' log-permeability field is a scaled version of the
71st slice of the log-permeability field data from the SPE10 model
\cite{SPE10}. Specifically, letting $l(x)$ be the SPE log-permeability data, 
we use $\theta_\text{true}(\bvec{x}) = 0.5\left[ l(\bvec{x})-\min(l(\bvec{x})) \right]+\min(l(\bvec{x}))$. 
Right:
the maximum a posteriori estimate of the permeability field obtained after
solving the Bayesian inverse problem governed by \eqref{eq:DarcyMixed} using
six pressure observations. The observation locations are shown as black dots.
}
\label{fig:spe10slice}
\end{figure}

As synthetic data for estimating $\theta$, we use six noisy pressure observations
obtained by solving \eqref{eq:DarcyMixed} using a scaled version of the 71st
slice of the SPE10 permeability model (shown in Figure~\ref{fig:spe10slice}).
We use a bi-Laplacian prior for $\theta$ (see~\cite{bui:infBayes}) and assume 
a Gaussian noise model.
The Bayesian inverse problem of estimating $\theta$ using
pressure measurements is a nonlinear inverse problem that in general requires
a sampling algorithm.  We compute an approximate solution
to that Bayesian inverse problem using a Laplace approximation to the posterior, 
which is a Gaussian approximation of the posterior centered at the 
maximum a posteriori probability (MAP) point; see
e.g.,~\cite{bui:infBayes}. 

To generate samples of the uncertain velocity field, we proceed as follows: we
generate log-permeability field realizations by drawing samples from the
approximate posterior for $\theta$; subsequently, we solve the pressure
equation to obtain the corresponding pressure fields, which are then used to
compute velocity field samples, $\bvecS{v}(\theta_i)$, using Darcy's Law
\eqref{eq:velocity}.  For illustration, four different Darcy velocity field
samples, together with the permeability fields, are shown in Figure~\ref{fig:samples}. 

\begin{figure}
\centering
\begin{tikzpicture}
\node (11) at (-5.8,4)
  {\includegraphics[width=0.5\linewidth]{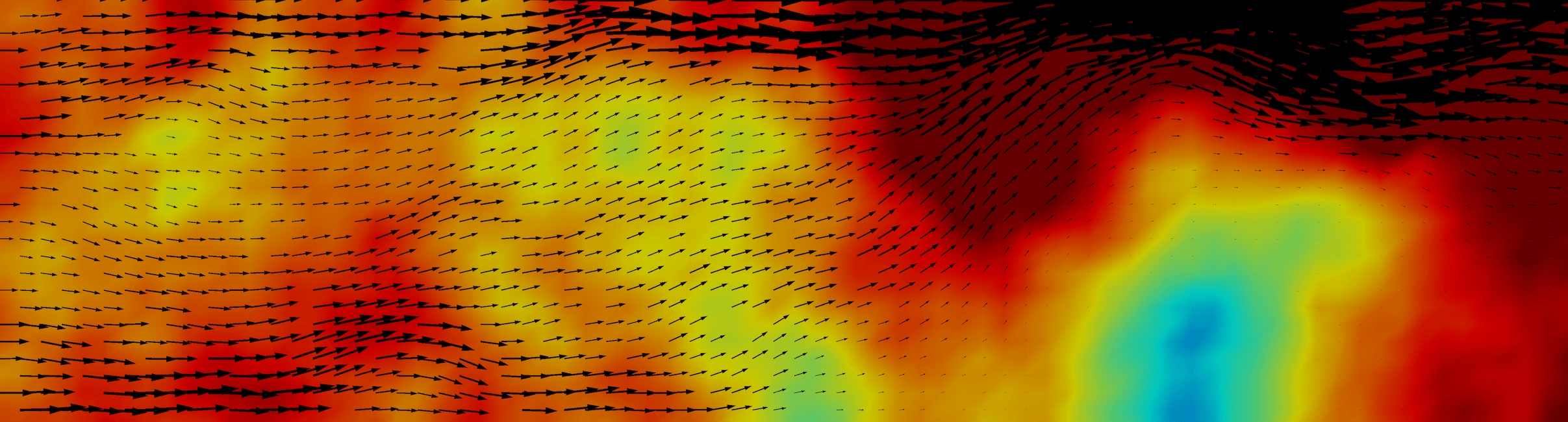}};
\node (12) at (2.5,4.0005)
  {\includegraphics[width=0.5\linewidth]{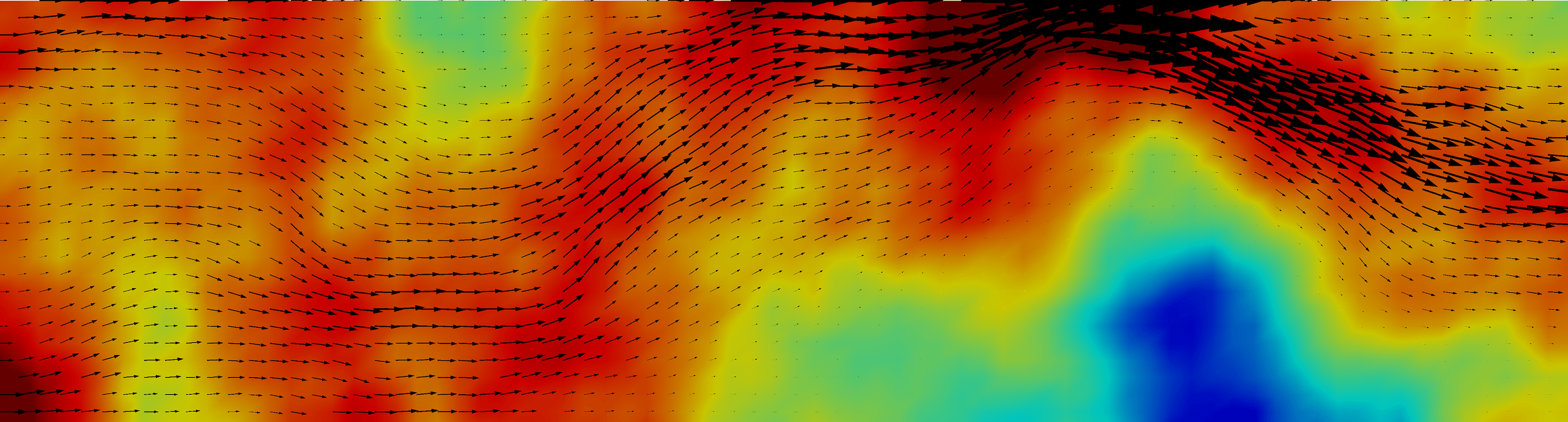}};
\node (21) at (-5.8,1.7)
  {\includegraphics[width=0.5\linewidth]{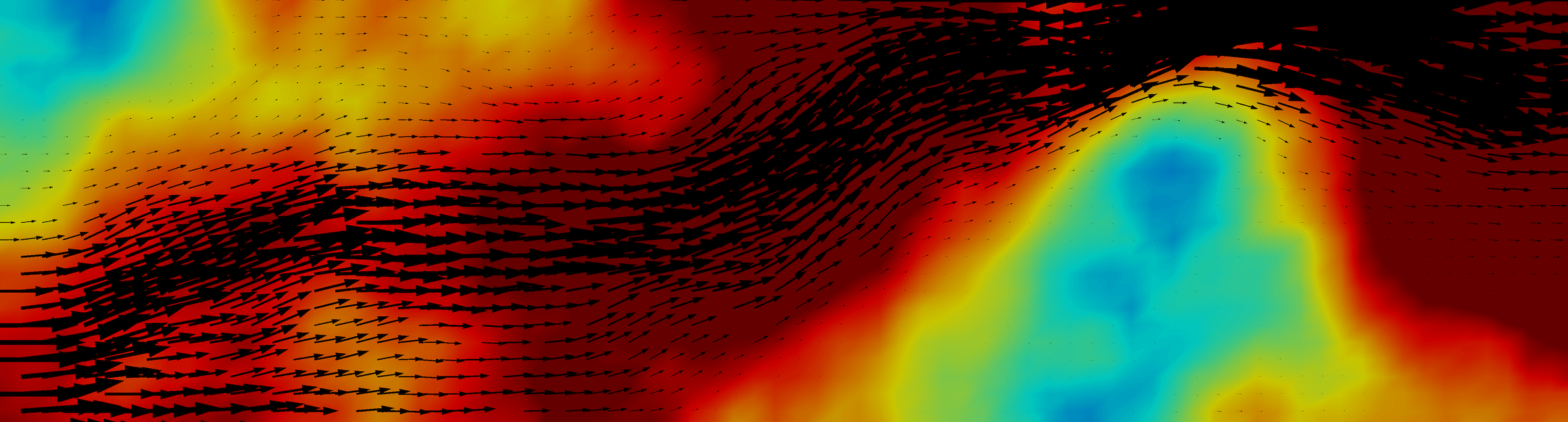}};
\node (22) at (2.5,1.699)
  {\includegraphics[width=0.5\linewidth]{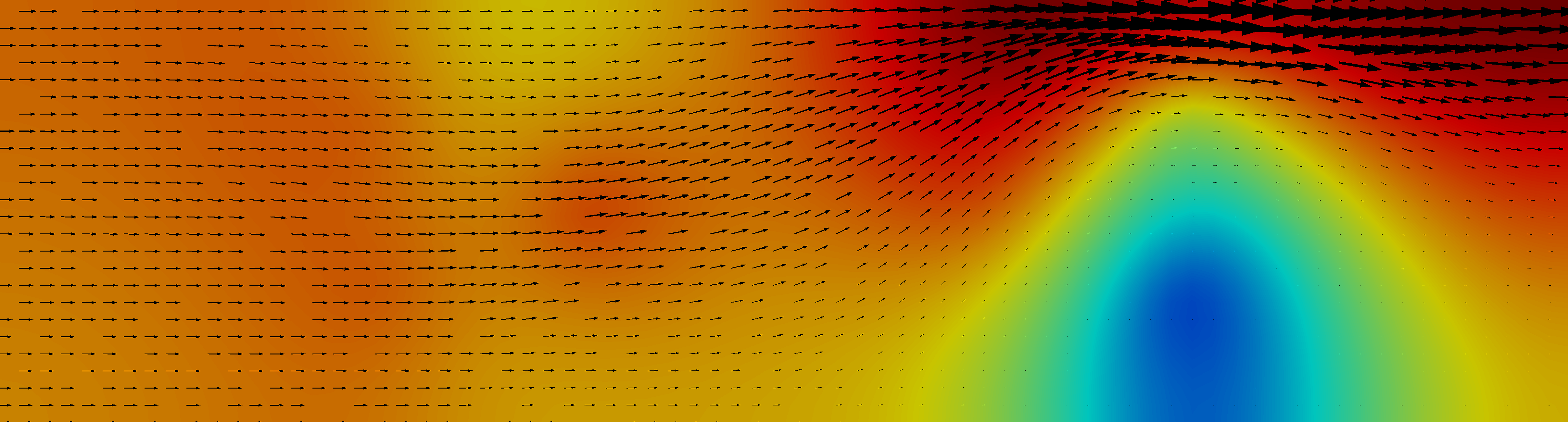}};
\end{tikzpicture}
\caption{Three log-permeability field samples $\theta_i$ and their corresponding velocity fields $\bvec{v}(\theta_i)$ (shown as black arrows) obtained from the posterior distribution $\pi_{\theta}$ and the
  log-permeability field and velocity field corresponding to the MAP point (lower
  right).}
\label{fig:samples}
\end{figure}

\subsection{Uncertainty in the initial time}
The second source of irreducible uncertainty we consider is the initial time
$T_0$. Receiving measurements of the concentration, one usually
does not know exactly \emph{when} the contaminant has been introduced
and thus we might only have an estimate for $T_0$. This uncertainty of
the initial time should be taken into account when computing optimal
designs, i.e., optimal designs should be tailored to a range of
initial times $T_0$. This uncertainty in the initial time is an
additional irreducible uncertainty we take into account, and we target
designs that are optimized for $T_0 \in [-1,1]$ and assume 
$T_0 \sim \mathcal{U}(-1,1)$.

\section{Subsurface flow example: optimal experimental design under uncertainty}\label{sec:numerics}
In this section we present numerical results for optimal sensor
placement under uncertainty and compare results obtained taking into
account the irreducible model uncertainty with results that are
computed for a fixed realization of the irreducible uncertainty.

\subsection{Setup of the OEDUU problem}\label{subsec:setup_numerics}

As explained above, we solve the optimal experimental design under
uncertainty (OEDUU) problem using sample average approximation. In our
numerical tests, we use 100 Monte Carlo samples of the irreducible model
uncertainty $\xi = (\theta,T_0)$, i.e., $N = 100$
in~\eqref{eq:disc_functional}. In table \ref{table:basisGrowth}, we show how many reduced
model basis vectors are needed for different tolerances. The moderate
growth in the number of needed ROM basis vectors indicates that 100
Monte Carlo samples capture a reasonable amount of the overall uncertainty.

%
%

\begin{table}[]
\centering
\begin{tabular}{|c|c|c|c|c|}
\hline
  \diagbox{$\mu$}{$N$}     & 50  & 100 & 300  & 500  \\ \hline
0.002  & 350 & 410 & 469  & 494  \\ \hline
0.0001 & 881 & 985 & 1051 & 1082 \\ \hline
\end{tabular}
\caption{Comparison of the number of basis vectors needed to compute
  $\M{Q}$ and $\widehat{\M{Q}}$ for different numbers $N$ of velocity
  samples (columns) and
  tolerances $\mu = 0.002$, $0.0001$ (rows). Here, we use
  Algorithm \ref{alg:composite_basis} with one cluster.
}
\label{table:basisGrowth}
\end{table}

As discussed in section~\ref{sec:mod_red}, minimization of the OEDUU
objective~\eqref{eq:disc_functional} is made tractable by elimination
of the PDE solves, throughout the optimization iterations, via
computation of a composite reduced order basis using
algorithm~\ref{alg:composite_basis}. 
In table~\ref{table:clustering}, we
study the accuracy and dimension of the joint basis for different
numbers of clusters and choices of $\mu$, which control the
accuracy of the reduced models.

Our clustering algorithm sorts the uncertain model parameters into bins based
on a distance measure~\eqref{eq:distance} that requires choosing a suitable
initial concentration $\bvec{m}$. Since contaminants in groundwater typically
originate from a few localized sources, 
we choose the initial
concentration $\bvec{m}$, in the definition of the distance measure~\eqref{eq:distance}, 
to be the sum of three radial basis functions
with varying centers, spreads and magnitudes. 
From table~\ref{table:clustering}, we note that for 100 samples of
the irreducible uncertainty $\xi$, clustering is not needed as the
resulting dimension of $\M{Q}$ and $\widehat{\M{Q}}$ in \eqref{eq:FQF}
is rather small. For smaller choices of $\mu$, such as $\mu =
10^{-4}$, or larger numbers of samples, clusters becomes more
important to reduce memory usage and compute time.

We find that the joint basis obtained using 1 cluster and $\mu = 0.002$
is sufficient for our purposes, both for approximation of the OEDUU objective
$\phiN$ in \eqref{eq:disc_functional} and approximation of $\widetilde{\M{F}}_i$ for each
forward operator sample.  Thus, this choice of parameters is used henceforth.

\begin{table}[]
\center
\begin{tabular}{|c|c|c|c|}
\hline
			     & {$\bvecS \mu$}    & \textbf{Number basis vectors} & \textbf{Error}    \\ \hline
\multirow{2}{*}{\textbf{1 Cluster}}  & $ 0.002$  & 410                           & $6.876\times 10^{-4}$              \\ \cline{2-4} 
			     & $ 0.0001$ & 985                          & $2.335\times 10^{-6}$             \\ \hline
\multirow{2}{*}{\textbf{4 Clusters}} & $ 0.002$  & 283,267,271,279               & $8.348\times 10^{-4}$              \\ \cline{2-4} 
			     & $ 0.0001$ & 667,751,710,731               & $7.943\times 10^{-6}$              \\ \hline
\end{tabular}
\caption{Comparison of ROMs obtained using algorithm \ref{alg:composite_basis} with $\mu = 0.002$, $0.0001$ and different numbers of clusters. Using $N=100$ sample velocity fields and initial times, $\phi_N(\bvec{w})$ in objective \eqref{eq:disc_functional} is
  computed for $\bvec{w} = [1,\ldots,1]$, i.e., the design which
  includes all sensor locations.
  The relative error in column four is computed as $ \frac{|{\phi_N}(\bvec{w}) -
    {\phi^t_N}(\bvec{w})|}{|{\phi^t_N}(\bvec{w})|}$, where the true
  value ${\phi^t_N}(\bvec{w})$
  was approximated using a reduced basis
  computed using algorithm~\ref{alg:composite_basis} with $\mu
  = 10^{-6}$. The third column shows how many basis vectors $k$ were
  needed in the algorithm to obtain the desired accuracy as described by line 6 in algorithm \ref{alg:composite_basis}.}
\label{table:clustering}
\end{table}

To obtain sparse and binary optimal designs, we solve a sequence of
minimization problems of the form~\eqref{eq:disc_functional}.  Following the
sparsification approach outlined in section~\ref{subsec:oedProb}, we use penalty
functions of the form \eqref{eq:weightedPenalty} with $\alpha = 0.1 $ and
$\varepsilon(i) = \left( \frac{2}{3}\right)^i$, $i = 1, 2, 3, \ldots$. For our problem setup, the weights converge to binary values in approximately 20 iterations of the sparsification procedure.
In our computations, we use
the Broyden-Fletcher-Goldfarb-Shannon (BFGS) method available in python's
\verb+scipy+ library.  We supply this minimization algorithm with the objective
function and its gradient as described in section~\ref{sec:discrete_oed}. 

\subsection{Solving the OEDUU problem}\label{subsec:comp}
Here, we demonstrate that solving the OEDUU problem with our proposed
approach is effective in producing designs for which the expected
value of the posterior uncertainty \eqref{eq:expected_val} is
small. In particular, we numerically verify that the SAA
\eqref{eq:finite_expectation} is a reasonable approximation.  This is
done by solving the OEDUU problem with $N=100$ SAA samples, and
comparing the expected value of the objective obtained using the optimal experimental designs under uncertainty, which we will refer to as \emph{uncertainty-aware designs}, 
with that obtained with \emph{deterministic designs}.  By
deterministic designs we mean optimal designs obtained using a 
single sample
from the irreducible uncertainty (velocity field and initial time);
this amounts to minimizing~\eqref{eq:disc_functional} with $N = 1$.

The results are shown in Figure~\ref{fig:super-plot1v2}. To obtain
designs with different numbers of sensors, \eqref{eq:disc_functional} was solved for different regularization parameters $\gamma$, which indirectly controls the sparsity of the optimal designs. To
approximate the expectation shown on the $y$-axis, we use a Monte
Carlo approximation of the expectation of the trace update $\mathbb{E}
\left[\tr\left( \M{K} (\xi,\bvec{w}) \right)\right]$ with 100 samples
from the irreducible uncertainty that are drawn independently from the
SAA samples used to compute the uncertainty-aware designs. Additionally, to avoid bias due
to the model reduction, we use more accurate reduced models,
specifically we use Algorithm \ref{alg:composite_basis} with
$\mu = 10^{-4}$.

If the sampling as well as the model reduction errors are small
enough, we would expect that the uncertainty-aware designs reduce the expected trace more
than deterministic designs. As can be seen in
Figure~\ref{fig:super-plot1v2}, this is the case most of the time,
i.e., we find good binary minimizers of the objective \eqref{eq:expected_val}.
Only very few deterministic designs
are superior to the uncertainty-aware designs, which is likely due to sampling
error.

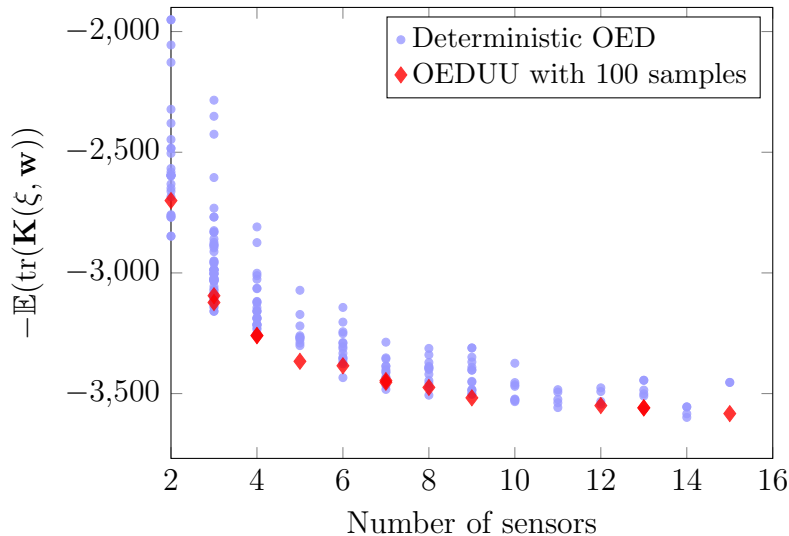
\begin{figure}[ht]\centering
\begin{tikzpicture}[]
\begin{axis}[compat=newest, width=8cm, height=6cm, scale only axis,
xlabel=Number of sensors,
ylabel=${-\mathbb{E}(\tr(\M{K}(\xi,\bvec{w}))}$,
xmin=2, xmax=16, ymax=-1900,
legend style={font=\small,nodes=right}]
\addplot [color=blue!40!white, mark=*, mark size=1.5pt, only marks,opacity=0.8]
table[x=nnz,y=trace] {data/100Samps_u_1_m1p1_FINAL.out};
\addlegendentry{Deterministic OED}
\addplot [color=red, mark=diamond*, mark size=3pt, only marks,opacity=0.8]
table[x=nnz,y=trace] {data/100Samps_u_100_m1p1_FINAL.out};
\addlegendentry{OEDUU with 100 samples}
\end{axis}
\end{tikzpicture}
\caption{Trace reduction ${\phi_N}$ in objective \eqref{eq:phiN} for
  deterministic OED (light blue dots) and OEDUU (red diamonds).  The
  irreducible uncertainty for the uncertainty-aware designs is approximated using 100
  samples from the irreducible uncertainty (velocity field and initial
  times). The deterministic designs are computed for 20 samples from
  the irreducible uncertainty.  The $x$-axis shows the number of
  sensors for each design.  The expectation in the objective, which is
  shown on the $y$-axis, is approximated using 100 samples that are
  chosen independently from the samples used to compute the
  designs.}
\label{fig:super-plot1v2}
\end{figure}

\subsection{Quality of the computed optimal designs}
Ultimately, our goal is to choose the sensors for data collection
which will allow us to learn the most about the initial concentration
for a fixed velocity field and initial time $T_0$.  Here we
demonstrate numerically that the designs computed via OEDUU perform
better (produce more informative data) on average than deterministic
designs computed for realizations of the irreducible model
uncertainty.

Here, we study the effectiveness of the same uncertainty-aware designs used above not in
expectation, but for 100 individual realizations $\xi_i$ of the
irreducible uncertainty, which again differ from the SAA samples used to
compute the uncertainty-aware designs.  Results are shown in Figure~\ref{fig:super-plot2}, where
now we show percentiles for the trace updates for individual $\xi_i$.

We observe that the designs obtained using OEDUU tend to have a
smaller mean in the update trace than the designs obtained using
individual $\xi_i$, particularly when we only use few sensors.
Additionally, the 25th--75th and 2nd--98th percentiles show that
poorly performing designs are less likely when one accounts for the
uncertainty in the design computation. This is again, particularly the
case for designs with small numbers of sensors. We believe that 
the
reason for diminishing benefit of computing designs under uncertainty for larger
number of sensors is that at each sensor location, 5 
measurements in time are used. Thus, most information about the initial
condition that can be recovered is already available
from a rather small number of sensors and thus different designs play a
less important role. Here, the diffusion contained
in the governing PDE plays an important role as well, as it
limits the resolution of the initial condition reconstruction that can
be obtained from observations.
This is also reflected by the decreasing gain
of using more than 5 sensors.

The benefit of computing uncertainty-aware designs is greater for inverse
problems where less information can be gained at each sensor location. This can
be seen in Figure~\ref{fig:super-plot2_lessT}. To obtain these results, we use
the same inverse problem and ROM setup as described in
Sections~\ref{sec:IP}-\ref{subsec:setup_numerics} but restrict the times at
which we can measure data to $\tau_i = 12,15$, effectively reducing the amount
of information gathered at sensors and increasing the importance of careful
sensor placement. Comparing Figures~\ref{fig:super-plot2} and
\ref{fig:super-plot2_lessT},  we can see that for ``harder'' inverse problems,
i.e., those with less-informative data, the difference between
uncertainty-aware designs and deterministic designs is more significant.

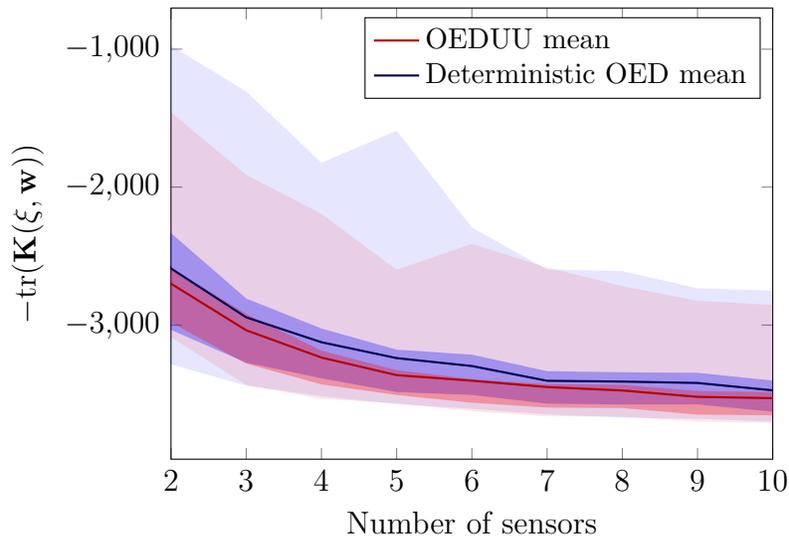
\begin{figure}[h]\centering
\begin{tikzpicture}[]
\begin{axis}[compat=newest, width=8cm, height=6cm, scale only axis,
xlabel=Number of sensors, xmin=2, xmax=10,
ylabel=${-\mbox{tr}(\mathbf{K}(\xi,\bvec{w}))}$,
legend style={font=\small,nodes=right}]
\addplot [color=red!70!black, mark=none, mark size=10.5pt,thick]
table[x=nnz,y=trace] {data/1Samp_u_100Samps_Mean.out};
\addlegendentry{OEDUU mean}

\addplot [color=blue!30!black, mark=none, mark size=10.5pt,thick]
table[x=nnz,y=trace] {data/1Samp_u_1Samp_Mean.out};
\addlegendentry{Deterministic OED mean}

\addplot [name path= low100_2,color=red!70!black, mark=none, mark size=0.8pt,thick,opacity=0.01]
table[x=nnz,y=trace] {data/1Samp_u_100Samps_2th.out};

\addplot [name path=low1_2, color=blue, mark=none, mark size=0.8pt,thick,opacity=0.01]
table[x=nnz,y=trace] {data/1Samp_u_1Samp_2th.out};

\addplot [name path=high100_98,color=red, mark=none, mark size=0.8pt,thick,opacity=0.01]
table[x=nnz,y=trace] {data/1Samp_u_100Samps_98th.out};

\addplot [name path=high1_98,color=blue, mark=none, mark size=0.8pt,thick,opacity=0.01]
table[x=nnz,y=trace] {data/1Samp_u_1Samp_98th.out};

\addplot [name path= low100_25,color=red!70!black, mark=none, mark size=0.8pt,thick,opacity=0.01]
table[x=nnz,y=trace] {data/1Samp_u_100Samps_25th.out};

\addplot [name path=low1_25, color=blue, mark=none, mark size=0.8pt,thick,opacity=0.01]
table[x=nnz,y=trace] {data/1Samp_u_1Samp_25th.out};

\addplot [name path=high100_75,color=red, mark=none, mark size=0.8pt,thick,opacity=0.01]
table[x=nnz,y=trace] {data/1Samp_u_100Samps_75th.out};

\addplot [name path=high1_75,color=blue, mark=none, mark size=0.8pt,thick,opacity=0.01]
table[x=nnz,y=trace] {data/1Samp_u_1Samp_75th.out};

\addplot [
thick,
color=blue!50!black,
fill=blue, 
fill opacity=0.1
]fill between[ 
of = low1_2 and high1_98, 
];
\addplot [
thick,
color=red!50!black,
fill=red, 
fill opacity=0.1
]fill between[ 
of = low100_2 and high100_98, 
];

\addplot [
thick,
color=blue!50!black,
fill=blue, 
fill opacity=0.3
]fill between[ 
of = low1_25 and high1_75, 
];
\addplot [
thick,
color=red!50!black,
fill=red, 
fill opacity=0.3
]fill between[ 
of = low100_25 and high100_75, 
];
\end{axis}
\end{tikzpicture}
\caption{Comparing the mean for the trace in the deterministic
  objective, $\phi_N$ with $N=1$ in \eqref{eq:phiN}, using
  deterministic designs and designs taking into account the
  uncertainty. The uncertainty-aware designs and the deterministic designs are the same as
  used in Figure~\ref{fig:super-plot1v2}. Each design is used to
  evaluate the deterministic objective for 100 realizations of the
  irreducible uncertainty. The sample mean is plotted as a solid line
  and the shaded regions depict the 25th--75th and 2nd--98th
  percentile envelopes (for both the mean and the envelopes, red is
  used for uncertainty-aware designs and blue for deterministic designs).}
\label{fig:super-plot2}
\end{figure}

\begin{figure}[h]\centering
\begin{tikzpicture}[]
\begin{axis}[compat=newest, width=8cm, height=6cm, scale only axis,
xlabel=Number of sensors, xmin=2, xmax=10, ymax=-800, ymin=-3600,
ylabel=${-\mbox{tr}(\mathbf{K}(\xi,\bvec{w}))}$,
legend style={font=\small,nodes=right}]
\addplot [color=red!70!black, mark=none, mark size=10.5pt,thick]
table[x=nnz,y=trace] {data/2Times_100_9999_Mean.out};
\addlegendentry{OEDUU mean}

\addplot [color=blue!30!black, mark=none, mark size=10.5pt,thick]
table[x=nnz,y=trace] {data/2Times_1_9999_Mean.out};
\addlegendentry{Deterministic OED mean}

\addplot [name path= low100_2,color=red!70!black, mark=none, mark size=0.8pt,thick,opacity=0.01]
table[x=nnz,y=trace] {data/2Times_100_9999_2th.out};

\addplot [name path=low1_2, color=blue, mark=none, mark size=0.8pt,thick,opacity=0.01]
table[x=nnz,y=trace] {data/2Times_1_9999_2th.out};

\addplot [name path=high100_98,color=red, mark=none, mark size=0.8pt,thick,opacity=0.01]
table[x=nnz,y=trace] {data/2Times_100_9999_98th.out};

\addplot [name path=high1_98,color=blue, mark=none, mark size=0.8pt,thick,opacity=0.01]
table[x=nnz,y=trace] {data/2Times_1_9999_98th.out};

\addplot [name path= low100_25,color=red!70!black, mark=none, mark size=0.8pt,thick,opacity=0.01]
table[x=nnz,y=trace] {data/2Times_100_9999_25th.out};

\addplot [name path=low1_25, color=blue, mark=none, mark size=0.8pt,thick,opacity=0.01]
table[x=nnz,y=trace] {data/2Times_1_9999_25th.out};

\addplot [name path=high100_75,color=red, mark=none, mark size=0.8pt,thick,opacity=0.01]
table[x=nnz,y=trace] {data/2Times_100_9999_75th.out};

\addplot [name path=high1_75,color=blue, mark=none, mark size=0.8pt,thick,opacity=0.01]
table[x=nnz,y=trace] {data/2Times_1_9999_75th.out};

\addplot [
thick,
color=blue!50!black,
fill=blue, 
fill opacity=0.1
]fill between[ 
of = low1_2 and high1_98, 
];
\addplot [
thick,
color=red!50!black,
fill=red, 
fill opacity=0.1
]fill between[ 
of = low100_2 and high100_98, 
];

\addplot [
thick,
color=blue!50!black,
fill=blue, 
fill opacity=0.3
]fill between[ 
of = low1_25 and high1_75, 
];
\addplot [
thick,
color=red!50!black,
fill=red, 
fill opacity=0.3
]fill between[ 
of = low100_25 and high100_75, 
];
\end{axis}
\end{tikzpicture}
\caption{Same as Figure~\ref{fig:super-plot2}, but with only two time
  observations at $\tau_i = 12,15$ (rather than five). In this regime
  with less observations, the uncertainty-aware design outperforms
  deterministic designs more significantly also for larger numbers of
  sensors.}
\label{fig:super-plot2_lessT}
\end{figure}
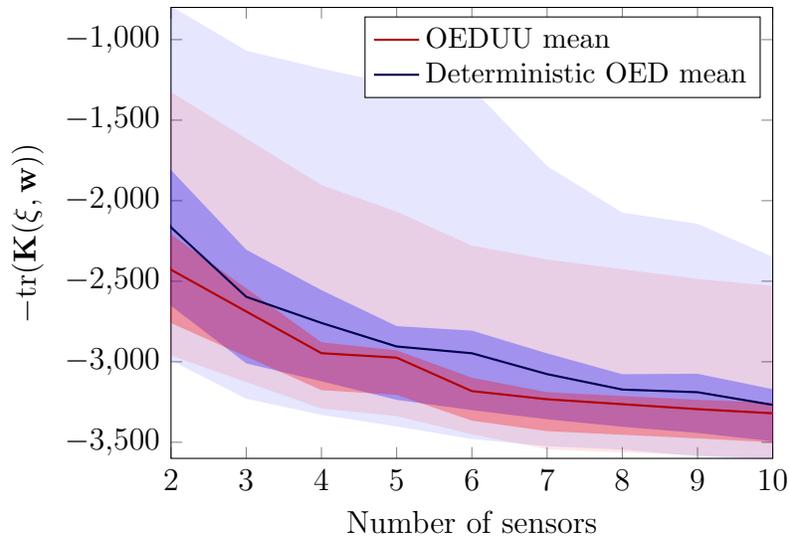

\section{Discussion and conclusions}\label{sec:conclusions}

We have developed a mathematical formulation and numerical scheme for
computing A-optimal experimental designs for infinite-dimensional
Bayesian linear inverse problems governed by PDEs with irreducible model
uncertainty.
The proposed measurement space approach replaces 
trace estimation in an infinite-dimensional (high-dimensional
upon discretization space) with trace estimation in the (finite-dimensional)
measurement space.
The computation of a \emph{joint} reduced basis capturing the action
of the forward operators for different random samples allows efficient computation of optimal designs.  Numerical experiments for the
inversion of initial concentration of a contaminant in groundwater
indicate that, on average, designs that take the model uncertainty into 
account are superior to those that do not. This superiority
is less pronounced as more sensors are being used.  

We have used a Monte Carlo approach for dealing with
uncertainty, which can require many samples for adequate resolution.
A possible extension of our work is to consider alternate approaches
for approximating the uncertainty, e.g., Taylor expansions of the
uncertainty or stochastic approximation (SA). In the latter, the samples for the
irreducible uncertainty are not chosen a priori, but are varied during
the optimization. This avoids potential bias of the design towards the
chosen samples but leads to several additional challenges in the
optimization. Another important research question is the generalization
of our approach to nonlinear inverse problems. One possibility is to follow the
framework in \cite{AlexanderianPetraStadlerEtAl16}, which is based on
Gaussian approximations at the MAP point. Clearly, nonlinear problems
under uncertainty become computationally rather expensive.

One aspect of OED under model uncertainty that is not explored in the
present work is that of dealing with \emph{reducible} sources of model
uncertainty, i.e., additional uncertainties that could be reduced
through observational data.
However, the focus might be on estimation of primary parameters of
interest and not on estimation of this (secondary) reducible
uncertainty. Hence, the design should be chosen to focus mainly on the
primary parameter and only on the secondary parameters to the extent
that it aids inference of the primary parameters.

\ack KK was supported in part by the Research Training Group in
Modeling and Simulation funded by the National Science Foundation
(NSF) via grant RTG/DMS \#1646339. GS acknowledges partial support from
the NSF grant \#1723211. The authors would like to thank 
Noemi Petra and Benjamin Peherstorfer for helpful discussions. \\

\bibliographystyle{plain}
\bibliography{my_biblio}

\begin{thebibliography}{10}

\bibitem{SPE10}
{SPE} comparative solution project.
\newblock \url{https://www.spe.org/web/csp/datasets/set02.htm}.

\bibitem{alexanderian:Doptimal}
Alen Alexanderian, Philip~J Gloor, Omar Ghattas, et~al.
\newblock On {B}ayesian {A}-and {D}-optimal experimental designs in infinite
  dimensions.
\newblock {\em Bayesian Analysis}, 11(3):671--695, 2016.

\bibitem{alexanderian:oed}
Alen Alexanderian, Noemi Petra, Georg Stadler, and Omar Ghattas.
\newblock A-optimal design of experiments for infinite-dimensional {B}ayesian
  linear inverse problems with regularized $\ell_0$-sparsification.
\newblock {\em SIAM Journal on Scientific Computing}, 36(5):A2122--A2148, 2014.

\bibitem{AlexanderianPetraStadlerEtAl16}
Alen Alexanderian, Noemi Petra, Georg Stadler, and Omar Ghattas.
\newblock A fast and scalable method for {A}-optimal design of experiments for
  infinite-dimensional {B}ayesian nonlinear inverse problems.
\newblock {\em SIAM Journal on Scientific Computing}, 38(1):A243--A272, 2016.

\bibitem{AlexanderianSaibaba18}
Alen Alexanderian and Arvind~K. Saibaba.
\newblock Efficient {D}-optimal design of experiments for infinite-dimensional
  {B}ayesian linear inverse problems.
\newblock {\em SIAM Journal on Scientific Computing}, 40(5):A2956--A2985, 2018.

\bibitem{AtkinsonDonev92}
Anthony~C. Atkinson and Alexander~N. Donev.
\newblock {\em Optimum Experimental Designs}.
\newblock Oxford, 1992.

\bibitem{attia2018goal}
Ahmed Attia, Alen Alexanderian, and Arvind~K Saibaba.
\newblock Goal-oriented optimal design of experiments for large-scale bayesian
  linear inverse problems.
\newblock {\em Inverse Problems}, 34(9):095009, 2018.

\bibitem{bear2018modeling}
Jacob Bear.
\newblock {\em Modeling phenomena of flow and transport in porous media},
  volume~31.
\newblock Springer, 2018.

\bibitem{BeyerSendhoff2007}
Hans-Georg Beyer and Bernhard Sendhoff.
\newblock Robust optimization---a comprehensive survey.
\newblock {\em Computer Methods in Applied Mechanics and Engineering},
  196(33):3190--3218, 2007.

\bibitem{bui:infBayes}
Tan Bui-Thanh, Omar Ghattas, James Martin, and Georg Stadler.
\newblock A computational framework for infinite-dimensional {B}ayesian inverse
  problems {P}art {I}: {T}he linearized case, with application to global
  seismic inversion.
\newblock {\em SIAM Journal on Scientific Computing}, 35(6):A2494--A2523, 2013.

\bibitem{candes2008enhancing}
Emmanuel~J Candes, Michael~B Wakin, and Stephen~P Boyd.
\newblock Enhancing sparsity by reweighted $\ell_1$ minimization.
\newblock {\em Journal of Fourier analysis and applications}, 14(5-6):877--905,
  2008.

\bibitem{daon2016mitigating}
Yair Daon and Georg Stadler.
\newblock Mitigating the influence of the boundary on {PDE}-based covariance
  operators.
\newblock {\em Inverse Problems {\&} Imaging}, 12(5):1083--1102, 2018.

\bibitem{dashti:bayesian}
Masoumeh Dashti and Andrew~M. Stuart.
\newblock {\em The Bayesian Approach to Inverse Problems}, pages 311--428.
\newblock Springer International Publishing, 2017.

\bibitem{golub2012matrix}
Gene~H Golub and Charles~F Van~Loan.
\newblock {\em Matrix computations}, volume~3.
\newblock JHU press, 2012.

\bibitem{HaberHoreshTenorio08}
Eldad Haber, Lior Horesh, and Luis Tenorio.
\newblock Numerical methods for experimental design of large-scale linear
  ill-posed inverse problems.
\newblock {\em Inverse Problems}, 24(055012):125--137, 2008.

\bibitem{HaberHoreshTenorio10}
Eldad Haber, Lior Horesh, and Luis Tenorio.
\newblock Numerical methods for the design of large-scale nonlinear discrete
  ill-posed inverse problems.
\newblock {\em Inverse Problems}, 26(2):025002, 2010.

\bibitem{HaberMagnantLuceroEtAl12}
Eldad Haber, Zhuojun Magnant, Christian Lucero, and Luis Tenorio.
\newblock Numerical methods for {A}-optimal designs with a sparsity constraint
  for ill-posed inverse problems.
\newblock {\em Computational Optimization and Applications}, pages 1--22, 2012.

\bibitem{HalkoMartinssonTropp11}
Nathan Halko, Per~Gunnar Martinsson, and Joel~A. Tropp.
\newblock Finding structure with randomness: {P}robabilistic algorithms for
  constructing approximate matrix decompositions.
\newblock {\em SIAM Review}, 53(2):217--288, 2011.

\bibitem{HoreshHaberTenorio10}
Lior Horesh, Eldad Haber, and Luis Tenorio.
\newblock {\em Optimal Experimental Design for the Large-Scale Nonlinear
  Ill-Posed Problem of Impedance Imaging}, pages 273--290.
\newblock Wiley, 2010.

\bibitem{HuanMarzouk13}
Xun Huan and Youssef~M. Marzouk.
\newblock Simulation-based optimal {B}ayesian experimental design for nonlinear
  systems.
\newblock {\em Journal of Computational Physics}, 232(1):288--317, 2013.

\bibitem{KorkelKostinaBockEtAl04}
Stefan K{\"o}rkel, Ekaterina Kostina, Hans~G. Bock, and Johannes~P.
  Schl{\"o}der.
\newblock Numerical methods for optimal control problems in design of robust
  optimal experiments for nonlinear dynamic processes.
\newblock {\em Optimization Methods \& Software}, 19(3-4):327--338, 2004.
\newblock The First International Conference on Optimization Methods and
  Software. Part II.

\bibitem{kouri2018optimization}
Drew~P Kouri and Alexander Shapiro.
\newblock Optimization of {PDE}s with uncertain inputs.
\newblock In {\em Frontiers in PDE-Constrained Optimization}, pages 41--81.
  Springer, 2018.

\bibitem{LongScavinoTemponeEtAl13}
Quan Long, Marco Scavino, Ra\'{u}l Tempone, and Suojin Wang.
\newblock Fast estimation of expected information gains for {B}ayesian
  experimental designs based on {L}aplace approximations.
\newblock {\em Computer Methods in Applied Mechanics and Engineering},
  259:24--39, 2013.

\bibitem{neitzel_sparse}
Ira Neitzel, Konstantin Pieper, Boris Vexler, and Daniel Walter.
\newblock A sparse control approach to optimal sensor placement in
  {PDE}-constrained parameter estimation problems.
\newblock {\em arXiv preprint arXiv:1905.01696}, 2019.

\bibitem{Pazman86}
Andrej P{\'a}zman.
\newblock {\em Foundations of Optimum Experimental Design}.
\newblock D. Reidel Publishing Co., 1986.

\bibitem{peherstorfer2014}
Benjamin Peherstorfer, Daniel Butnaru, Karen Willcox, and Hans-Joachim
  Bungartz.
\newblock Localized discrete empirical interpolation method.
\newblock {\em SIAM Journal on Scientific Computing}, 36(1):A168--A192, 2014.

\bibitem{Pukelsheim93}
Friedrich Pukelsheim.
\newblock {\em Optimal Design of Experiments}.
\newblock John Wiley \& Sons, New-York, 1993.

\bibitem{roininen2014whittle}
Lassi Roininen, Janne~MJ Huttunen, and Sari Lasanen.
\newblock Whittle-{M}at{\'e}rn priors for {B}ayesian statistical inversion with
  applications in electrical impedance tomography.
\newblock {\em Inverse Probl. Imaging}, 8(2):561, 2014.

\bibitem{ruthotto2018optimal}
Lars Ruthotto, Julianne Chung, and Matthias Chung.
\newblock Optimal experimental design for inverse problems with state
  constraints.
\newblock {\em SIAM Journal on Scientific Computing}, 40(4):B1080--B1100, 2018.

\bibitem{Sahinidis04}
Nikolaos~V Sahinidis.
\newblock Optimization under uncertainty: state-of-the-art and opportunities.
\newblock {\em Computers \& Chemical Engineering}, 28(6):971--983, 2004.

\bibitem{ShapiroDentchevaRuszczynski09}
Alexander Shapiro, Darinka Dentcheva, and Andrezj Ruszczynski.
\newblock {\em Lectures on Stochastic Programming: Modeling and Theory}.
\newblock Society for Industrial and Applied Mathematics, 2009.

\bibitem{Ucinski05}
Dariusz Uci{\'n}ski.
\newblock {\em Optimal measurement methods for distributed parameter system
  identification}.
\newblock CRC Press, Boca Raton, 2005.

\bibitem{villa2019hippylib}
Umberto Villa, Noemi Petra, and Omar Ghattas.
\newblock {h{IPPYlib}: An Extensible Software Framework for Large-Scale Inverse
  Problems Governed by PDEs; Part I: Deterministic Inversion and Linearized
  {B}ayesian Inference}.
\newblock {\em arXiv preprint arXiv:1909.03948}, 2019.

\bibitem{yu2018scalable}
Jing Yu and Mihai Anitescu.
\newblock Multidimensional sum-up rounding for integer programming in optimal
  experimental design.
\newblock {\em Mathematical Programming}, pages 1--40, 2017.

\end{thebibliography}

\end{document}